\documentclass[a4paper]{amsart} 

\usepackage{amsmath,amsthm,amssymb,amsfonts,mathrsfs,color,hyperref, mathtools, graphicx, enumitem, todonotes, soul}
\usepackage[]{geometry}

\theoremstyle{plain}
\begingroup
\newtheorem{theorem}{Theorem}
\newtheorem*{theorem*}{Theorem}
\newtheorem*{"theorem"}{``Theorem''}

\newtheorem{lemma}[theorem]{Lemma}
\endgroup

\theoremstyle{definition}
\begingroup

\endgroup

\theoremstyle{remark}
\begingroup
\newtheorem{remark}[theorem]{Remark}

\endgroup 

\numberwithin{equation}{section}
\newcommand{\R}{\mathbb R} 
\newcommand{\B}{{\mathcal B}}

\newcommand{\dist}{{\rm dist}}

\newcommand{\id}{\mathrm{id}}
\newcommand{\sign}{\mathrm{sign}}

\renewcommand{\H}{{\mathcal H}}

\renewcommand{\L}{{\mathcal L}}

\newcommand{\G}{{\mathcal G}}

\newcommand{\I}{{\mathcal I}}

\newcommand{\LRa} {\Leftrightarrow}
\newcommand{\Ra} {\Rightarrow}

\newcommand{\embeds}{\xhookrightarrow{\quad}}

\renewcommand{\d}{\,\mathrm{d}}

\newcommand{\dx}{\,\mathrm{d}x}

\newcommand{\dz}{\,\mathrm{d}z}
\newcommand{\ds}{\,\mathrm{d}s}
\newcommand{\dt}{\,\mathrm{d}t}

\newcommand{\eps}{\varepsilon}
\newcommand{\average}{{\mathchoice {\kern1ex\vcenter{\hrule height.4pt
width 6pt depth0pt} \kern-9.7pt} {\kern1ex\vcenter{\hrule
height.4pt width 4.3pt depth0pt} \kern-7pt} {} {} }}

 \allowdisplaybreaks

\xdefinecolor{dianablue}{rgb}{0.18,0.24,0.31}
\xdefinecolor{darkblue}{rgb}{0.1,0.1,0.7}
\xdefinecolor{darkgreen}{rgb}{0,0.5,0}
\xdefinecolor{darkgrey}{rgb}{0.35,0.35,0.35}
\xdefinecolor{darkorange}{rgb}{0.8,0.5,0}
\xdefinecolor{darkred}{rgb}{0.7,0,0}
\definecolor{darkgreen}{rgb}{0,0.6,0}

\definecolor{pittblue}{rgb}{0  0.2078  0.5804}

\begin{document}

\title[Neural Networks and Energy Gaps in SciML]{The Barron-Lipschitz Energy Gap and Depth Separation Phenomena in Scientific Machine Learning}

\author{Nima Rezaei}
\address{Nima Rezaei\\
Department of Mathematics\\
University of Pittsburgh\\
Pittsburgh, PA 15221, USA
}
\email{NIR100@pitt.edu}

\author{Stephan Wojtowytsch}
\address{Stephan Wojtowytsch\\
Department of Mathematics\\
University of Pittsburgh\\
Pittsburgh, PA 15221, USA
}
\email{s.woj@pitt.edu}

\date{\today}

\subjclass[2020]{46E35, 49Q20, 68T07, 74K20}
\keywords{Barron space, Lavrentiev gap, scientific machine learning, approximation theory}

\begin{abstract}
We illustrate in several examples that even neural networks of infinite width (specifically, Barron functions) may encounter substantial obstacles when used as a model class for problems in the calculus of variations. 
An instance of practical relevance concerns the bending, stretching and folding of a thin elastic shell with anchored or clamped boundary conditions where elastic energy could be reduced by folding along a circular line, but the neural networks can only describe straight folds along entire lines.
Conversely, we show that there is no gap between the energy that Barron functions and Lipschitz functions can achieve for a large class of integral first-order functionals.
\end{abstract}

\maketitle


\section{Introduction}
A rapidly growing corpus of literature illustrates a gap between the behavior of `high regularity' and `low regularity' solutions to various problems in fields as disparate as isometric embeddings \cite{nash1954c, nash1956imbedding, conti2010h}, fluid mechanics \cite{constantin1994onsager,  buckmaster2017onsager, de2019turbulence}, and the calculus of variations \cite{lavrentieff1927quelques, mania1934sopra, ball1982discontinuous, loewen1987lavrentiev, zhikov1987averaging, zhikov1995lavrentiev, foss2003lavrentiev, mazowiecka2017lavrentiev, balci2020new}. These differences are characterized by a transition from rigidity (high regularity) to flexibility (low regularity) in the first two fields and by a `Lavrentiev gap' in the third, where a physical energy functional can attain a strictly lower value in a class of less regular functions (for instance, releasing elastic tension by cavitation or fracture). Less obvious examples exist where the low regularity candidates remain continuous. In such situations, any numerical scheme based on `high regularity' approximating functions is likely to miss physically relevant candidates of low regularity. In particular, this includes numerical methods based on approximation by neural networks with sufficiently smooth activation functions, including all activation functions which allow for gradient-based loss minimization.

Crucially, the existence of a (finite, positive) {\em quantitative} gap follows from merely {\em qualitative} information on function regularity. In this work, we demonstrate that such a gap exists not just between classical smoothness classes, but also between the class of Lipschitz functions and the novel `Barron spaces', which are designed to model arbitrarily wide shallow neural networks with ReLU activation. Thus, the use of neural networks as a computational tool may introduce additional difficulty beyond the smoothness class in which the networks are naturally placed. 

More specifically, we construct variational problems which admit solutions in more general function classes (or classes of deeper neural networks in one practically relevant case), but not networks with a single hidden layer. In this sense, our results can also be considered through the lens of `depth separation' \cite{eldan2016power, telgarsky2016benefits, venturi2022depth} for scientific machine learning. Conversely, we demonstrate that for a large class of functionals, there exists no energy gap between the class of Lipschitz functions and Barron functions.
In particular, in those settings there is no dramatic qualitative jump in approximation power between shallow and deep networks, although quantitatively, approximation may be easier for particular architectures. Our main results are the following:

\begin{enumerate}
\item There is no Barron-Lipschitz gap for a large class of integral first-order functionals with or without imposed boundary condition (Theorem \ref{theorem no barron-lipschitz gap}).

\item There can be a Barron-Lipschitz gap for first-order functionals which include an $L^\infty$-norm or a sufficiently singular integrand in one dimension (Theorems \ref{theorem mania style example} and \ref{theorem second mania style example}). This can yield functions of almost minimal energy in Barron class to differ significantly from those in the Lipschitz class.

\item There is a Barron-Lipschitz energy gap for certain elastic energy functionals modeling the stretching, bending and folding of thin sheets (Theorem \ref{theorem bending}).
\end{enumerate}

These results contrast with more well-behaved energies for physics-informed neural network (PINN) PDE solvers, where solutions are generally not in Barron space if the boundary data would allow it, but can be approximated to high fidelity by Barron functions of low norm \cite{deepritz_regularity} -- see also \cite{vaishampayan2024solving} for analogous results with activation function ReLU$^k$ for $k\in \mathbb N$.

\section{A brief review of Barron spaces}

Barron spaces $\B$ are function spaces tailored specifically for approximation by shallow neural networks with ReLU activation. They contain not only all finite networks $f(x) = \sum_{i=1}^n a_i\,\sigma(w_i^Tx+b_i)$ with $\sigma(z) = \max\{z,0\}$, but also their `infinite width limits' $f(x) = \int_{\R^d\times \R} \sigma(w^Tx+b)\d\mu$ under a suitable norm bound assumption/moment bound assumption on the total variation measure $|\mu|$. These spaces have favorable behavior in terms of approximation by finite neural networks and estimation of integral quantities by few data points. In this article, we only require the following facts.

\begin{enumerate}
\item $\B(\R^d)$ is a Banach space for any $d$ and if $A\in \R^{d\times d}$ is a matrix, then $\|f\circ A\|_\B \leq \|f\|_\B\|A\|_{op}$.
\item Every Barron function is Lipschitz-continuous and $[f]_{Lip} \leq [f]_\B\leq \|f\|_\B$ \cite[Theorem 3.3]{barron_new} where $[f]_\B$ is the `Barron semi-norm' and $\|f\|_\B$ is the full Barron norm.
\item If $u\in C^\infty(\R^d)$, then for every $R>0$ there exists $u_R\in \B$ such that $u_R \equiv u$ in $B_R(0)$ \cite[Theorem 3.2]{barron_new}.
\item Any function $f \in \mathcal{B}(\mathbb{R}^d)$ can be written as a countable sum $f = \sum_{i=0}^\infty f_i,$
where
\begin{itemize}
\item $f_0 \in \B(\R^d)$ is $C^1$-smooth,
\item $f_i(x) = g_i(P_i x + b_i)$, where
	\item $P_i : \R^d \to \R^{k_i}$ is an orthogonal projection for $1 \leq k_i \leq d$ (i.e.\ $P_i P_i^T = I_{k_i \times k_i}$),
	\item $g_i \in \B(\R^{k_i})$ is $C^1$-smooth except at $0 \in \R^{k_i}$.
\end{itemize}
The structure theorem for Barron functions can be found in \cite[Theorem 5.9]{barron_new}.
\item In one dimension, for $u:(-1,1)\to\R$ there exists a Barron function $U:\R\to\R$ such that $U=u$ on $(-1,1)$ if and only if $u$ is weakly differentiable and $u'\in BV(-1,1)$ is a function of bounded variation \cite[Section 4.1]{barron_new}.
\end{enumerate}

In any dimension, Barron functions satisfy $\nabla u\in BV_{loc}(\R^d;\R^d)$. In particular, Barron functions admit a measure-valued Hessian. While the argument is straightforward, we believe that this has not been observed previously and prove it in the Appendix for the reader's convenience.

\begin{lemma}\label{lemma bounded variation gradient}
Let $u\in \B(\R^d)$. Then $\nabla u \in BV_{loc}(\R^d; \R^d)$ and the measure norm of the Hessian satisfies $\|D^2 u\|_{\mathcal M(B_R)} \leq \omega_{d-1}R^{d-1}\,\|u\|_\B$.
\end{lemma}

Much more thorough introductions can be found in \cite{ma2022barron, barron_new, wojtowytsch2022optimal}. We note that the `Barron spaces' under consideration in this note are also referred to as `representational Barron spaces' in the literature to distinguish them from the `spectral Barron class' into which they embed \cite[Theorem 3.1]{barron_new}. They are closely related to the function class $\mathcal F_1$ \cite{bach2017breaking} and Radon-BV spaces \cite{ongie2019function, parhi2021banach, unser2023ridges}, which mostly differ depending on whether the bias term is controlled in magnitude or not, i.e.\ whether the full Barron norm or Barron semi-norm is used. While the theories are similar, some surprising differences can be observed \cite{wojtowytsch2022optimal, boursier2023penalising, park2023minimum}.

We give a stronger version of the structure theorem for Barron functions \cite[Theorem 5.9]{barron_new} in spatial dimension $d=2$. As the proof is essentially disjoint from the remainder of the article, we postpone it to Appendix \ref{appendix barron structure}, along with the proof of Lemma \ref{lemma bounded variation gradient}.

\begin{theorem}\label{theorem structure barron sbv2}
Let $\Omega\subseteq \R^2$ be bounded, $f:\Omega \to\R$ be a Barron function and $K$ a relatively closed set of locally finite Hausdorff measure $\mathcal H^1$ such that the distributional Hessian of $f$ in $\Omega\setminus K$ satisfies $D^2f \in L^2(\Omega\setminus K)$. Then there exists a countable, locally finite union of lines $K' = \bigcup_{k\in\mathbb Z}L_k$ such that $K' \subseteq K$ and $D^2f \in L^2(\Omega\setminus K')$. 
\end{theorem}

The theorem statement is complicated by the fact that the set $K$ can be replaced by a larger closed set of finite Hausdorff measure, so the structure only holds for the {\em minimal} such set.

\section{No Barron-Lipschitz gap for integral first-order functionals}

Our first result is positive. For a large class of functionals, some of which do exhibit a positive gap between $W^{1,1}$ and $W^{1,\infty}$ (e.g.\ the example of \cite{mania1934sopra}), there exists no additional gap between the smaller class of Barron functions and the intermediate class of Lipschitz functions.

\begin{theorem}\label{theorem no barron-lipschitz gap}
Let $\Omega\subseteq\R^d$ be a bounded domain, $f:\overline \Omega\times \R\times \R^d \to \R$ and $g:\partial\Omega\times\R\times \R^d\to\R$ continuous functions. Further, let $1\leq q < \infty$ and $\eps>0$.

\begin{enumerate}
\item Assume that $\Omega$ has finite perimeter or $g\equiv 0$. If $w\in C^{0,1}(\overline\Omega)$, then there exists $w_\eps\in \B$ such that  $\|w-w_\eps\|_{W^{1,q}(\Omega)} \leq \eps$ and $\big|J(w_\eps) - J(w)\big| < \eps$ where
\[
J(u) = \int_\Omega f(x,u, Du)\dx + \int_{\partial\Omega} g(x,u)\d A.
\]

\item If $w\in C^{0,1}(\overline\Omega)$ and there exists a Barron function $\phi \in \B$ such that $\phi \equiv w$ on $\partial\Omega$, then there exists $w_\eps\in \B$ such that $w_\eps = w$ on $\partial\Omega$, $\|w-w_\eps\|_{W^{1,q}(\Omega)} + \|w-w_\eps\|_{C^0(\overline\Omega)} \leq \eps$ and $\big|J(w_\eps) - J(w)\big| < \eps$ where
\[
J(u) = \begin{cases} \int_\Omega f(x,u, Du)\dx & \text{if }u\equiv \phi \text{ on }\partial\Omega \\ + \infty &\text{else.}\end{cases}
\]
\end{enumerate}
In particular, we find in both cases that $\inf_{u\in\B}J(u) = \inf_{u\in W^{1,\infty}(\Omega)}J(u)$.
\end{theorem}

All bounded open sets with Lipschitz boundary (in particular, all bounded convex open sets) have finite perimeter, so the theorem applies in large generality. The expression $\d A$ denotes the integral with respect to the area measure $\|\nabla 1_\Omega\|$, which coincides with the $d-1$-dimensional Hausdorff measure on $\partial\Omega$ for regular boundaries. The conditions on $f, g$ can be relaxed, see Remark \ref{remark generalization}.

\begin{proof}
{\bf First claim.} Since $u$ is Lipschitz continuous on $\overline{\Omega}$, by Kirszbraun’s Theorem \cite[Theorem 7.2]{maggi2012sets} there exists an extension $U : \mathbb{R}^d \to \mathbb{R}$ such that $U = u$ on $\overline{\Omega}$ and $[U]_{Lip(\R^d)} = [u]_{Lip(\overline\Omega)}$. Moreover, by defining $\tilde{u}(x) = \max\{U(x), \min u\}$ and $ \tilde{U}(x) = \min\{\tilde{u}(x), \max u\}$, we find that there exists an extension $\tilde{U}$ (which we still denote by $U$) satisfying $U = u \text{ on } \overline{\Omega},$ $\min U = \min u,$ and  $\max U = \max u.$

Recall that $C^{0,1}(\overline\Omega)$ embeds into $W^{1,\infty}(\Omega)$ by Rademacher's theorem \cite[Chapter 3.1]{evans2025measure}. In particular, the distributional gradient of $U$ is given by a measurable bounded vector function. Let  $U_\eps = U \ast \eta_\eps,$ where $\eta$ is a standard mollifier.  By  \cite[Section 5.3, Theorem 1]{evans2022partial}, and  \cite[Appendix C, Theorem 7]{evans2022partial}, we see that $U_\eps \to U$ both uniformly (i.e.\ in $C^0(\overline \Omega))$ and in $W^{1,q}(\Omega)$. Furthermore, the bounded Lipschitz function $U$ lies in $W^{1,\infty}(\R^d)$ with $\|\nabla U\|_{L^\infty} = [U]_{Lip(\R^d)} = [u]_{Lip(\overline\Omega)}$ \cite[Theorem 1.41]{weaver2018lipschitz} and, as shown within the proof of  \cite[Section 5.3, Theorem 1]{evans2022partial}, we have $\nabla  U_\eps = \nabla  U \ast \eta_\eps.$ We conclude by \cite[Theorem 3.11]{rudin1987real} that there exists a subsequence (still denoted by $\{U_\eps\}_{\eps>0}$) such that $\nabla U_\eps \to \nabla U$ almost everywhere.
We further note that
\begin{align*}
\left|U_\eps(x)\right| &= \left| \int_{B_\eps(0)} \eta_\eps(y) \, U(x-y) \d y \right| \leq \|U \|_{L^{\infty}} = \|u \|_{L^{\infty}}\\
\left|\nabla U_\eps(x)\right| &= \left| \eta_\eps  \ast \nabla  U  (x) \right| = \left| \int_{B_\eps(0)} \eta_\eps(y) \, \nabla U(x-y) \d y \right|  \leq \|\nabla U\|_{L^\infty} = [u]_{Lip}
\end{align*}
for all $\eps>0$ and $x\in\R^d$. In other words,  $\|U_\eps\|_{L^\infty(\Omega)} \leq \|u\|_{L^\infty(\Omega)}$ and $\|\nabla U_\eps\|_{L^\infty(\Omega)} \leq [u]_{Lip}.$
As $f$ is continuous, we have:
\begin{align*}
	\left| f(x, U_\eps, \nabla U_\eps) \right| \leq \max_{\big( z \in \overline\Omega, |u| \leq M, \|\xi\|\leq M \big) } |f(z,u,\xi)| < \infty \qquad \forall x\in \overline \Omega.
\end{align*}
Since $\overline\Omega$ has finite measure, by the Dominated Convergence Theorem 
\[\int_{\Omega} f(x, U_\eps, \nabla U_\eps) \d x \to \int_{\Omega} f(x, u, \nabla u) \d x.
\]
Finally, as $U_\eps \to u$ uniformly on $\partial \Omega$  and $g$ is continuous on $\partial \Omega \times [-M,M]$, we see that also $g(x,U_\eps)\to g(x,u)$ uniformly. Since $\Omega$ has finite perimeter or $g\equiv 0$, this means that
\[
\int_{\partial \Omega} g(x,U_\eps) \, dA \to \int_{\partial \Omega} g(x,u) \, dA.
\]
Combining these two, we have shown that $J(U_\eps) \to J(u)$ for some sequence $\eps\to 0^+$. Finally, since $U_\eps\in C^\infty(\R^d)$ by  \cite[Appendix C, Theorem 7]{evans2022partial} and $\Omega$ is bounded, we can find $u_\eps \in \B$ such that $u_\eps = U_\eps$ on $\overline{\Omega}$ by \cite[Theorem 3.1]{barron_new}. Replacing $U_\eps$ by such $u_\eps \in \B,$ we have shown that there exists a sequence $u_\eps \in \B$ such that $u_\eps\to u$ both uniformly and in $W^{1,q}(\Omega)$ for all $q<\infty$ and $J(u_\eps) \to J(u).$

{\bf Second claim.} By assumption, we can decompose $w$ as
\[
w(x) = \big(w-\phi\big)(x) + \phi(x)
\]
with $\phi \in \B$ such that $w-\phi \equiv 0$ on $\partial\Omega$. Additionally, we know that  $[w-\phi]_{Lip} \leq [w]_{Lip} + [\phi]_{Lip} \leq [w]_{Lip} + [\phi]_\B$. In particular:
\[
\big|(w-\phi)(x)\big| \leq \big( [w]_{Lip} + [\phi]_\B\big) \, \dist(x, \partial\Omega).
\]
We define 
$
h_\eps(z) = \sign(z) \max \{|z| - \eps, 0\}
$
and $\partial \Omega^\eps = \{x \in \Omega: \dist(x, \partial\Omega) \leq \frac{\eps}{[w]_{Lip} + [\phi]_\B +1} \}.$ Now, we see that when $x \in \partial \Omega^\eps,$  we have:
\[
\big|(w-\phi)(x)\big| \leq \big( [w]_{Lip} + [\phi]_\B\big) \dist(x, \partial\Omega) \leq \eps \qquad \Ra\quad h_{\eps}(w-\phi)(x) = 0.
\]
We now define $\psi_\eps = h_\eps (w - \phi) + \phi$ and
 claim that as $ \eps \to 0,$ $\psi_\eps \to w$ in $W^{1,p} (\Omega)$ for all $p < \infty$.
First, note that
\[
\|\psi_\eps-w\|_{L^\infty(\Omega)} = \big\|\phi + h_\eps(w-\phi) - \phi - (w-\phi)\big\|_{L^\infty(\Omega)} = \big\|(h_\eps - \id)(w-\phi)\big\|_{L^\infty(\Omega)} \leq \|h_\eps - \id\|_{L^\infty(\R)} = \eps.
\]
where $\id$ denotes the identity map $\id(z) = z$ on $\R$. Now, we consider the convergence of gradients $\nabla h_\eps (w - \phi) \to \nabla (w-\phi)$ as $\eps \to 0.$ To this end, we observe that 
\begin{align*}
h_\eps (z) = \big(z - \eps\big)^+ - \big(z + \eps\big)^-
\end{align*}
where $z_+ = \max\{z, 0\}$ and $z_- = \max\{-z,0\}$.
We claim that for any $u \in W^{1,q}(\Omega)$ we have: $ h_\eps (u)  \in W^{1,q}(\Omega)$  and  $\nabla h_\eps (u) \to \nabla u,$ almost everywhere. By \cite[Part 3, Theorem 4.4]{evans2025measure} we get $\big(u - \eps\big)^+, \big(u + \eps\big)^- \in W^{1,q}(\Omega)$ and we have
\[
\nabla \big(u - \eps\big)^+ = \begin{cases} \nabla u & u>\eps\\ 0 & u \leq \eps \end{cases}
\qquad\text{and}\quad
\nabla \big(u + \eps\big)^- = \begin{cases} 0 & u \geq -\eps\\ - \nabla u & u < -\eps \end{cases}
\]
almost everywhere. Hence
\begin{align} \label{heps}
\nabla h_\eps (u) = \begin{cases}
\nabla u & |u| > \eps \\ 0 & |u| \leq \eps
\end{cases} \quad\xrightarrow{\eps\to0^+} \quad\begin{cases}
\nabla u & u\neq 0 \\ 0 & u=0
\end{cases}
\quad = \nabla u
\end{align}
pointwise almost everywhere, using that $\nabla u = 0$ almost everywhere on a level set (see for instance \cite[Part 4, Theorem 4.4]{evans2025measure}).
As an immediate consequence, we obtain: $\nabla \psi_\eps \to \nabla w$ almost everywhere. Additionally, by \eqref{heps} we can see that
\[
|\nabla \psi_\eps - \nabla w|^q = |\nabla h_\eps(w-\phi) - \nabla (w-\phi)|^q \leq |\nabla(w-\phi)|^q,
\] 
This is because when $\left|\left(w - \phi\right)\right| > \eps,$ $\nabla h_\eps (  w - \phi )  =0,$ thus  $\left| \nabla h_\eps (  w - \phi )  - \nabla (w - \phi) \right| = \left|\nabla ( w- \phi )\right|,$ and when $\left|\left(w - \phi\right)\right| \leq \eps,$ $\nabla h_\eps (  w - \phi )  = \nabla (w-\phi),$ thus  $\left| \nabla h_\eps (  w - \phi )  - \nabla (w - \phi) \right| = 0.$ We have:
\begin{align*}
  |\nabla(w-\phi)|^q \leq \big(|\nabla \phi| + |\nabla w|\big)^q \leq \big([w]_{Lip} + \|\phi\|_\B\big)^q.
\end{align*} 
Applying the Dominated Convergence Theorem and using that $\Omega$ has finite measure, we immediately conclude that $\nabla \psi_\eps \to \nabla w$ in $L^q(\Omega).$ As we have already shown that $\psi_\eps \to w$ in $L^{\infty}(\Omega)$ and using the continuous embedding $L^\infty(\Omega) \to L^q(\Omega)$ for the finite measure set $\Omega$ that $\psi_\eps \to w$ in $L^{q}(\Omega)$, we overall find that $\psi_\eps \to w$ in $W^{1,q}(\Omega).$

It remains to approximate $\psi_\eps$ by Barron functions that satisfy the boundary conditions. As $h_\eps(w-\phi) = 0$ in a neighborhood of $\partial\Omega$, we may extend it by zero to $\R^d\setminus \Omega$ without increasing the Lipschitz constant. 
Now, if $x\in\partial\Omega$ and $\delta < \frac{\eps}{[w]_{Lip}+[\phi]_\B + 1}$, then $h_\eps(w-\phi) \equiv 0$ on $B_\delta(x)$ since
\[
\dist(z, \partial\Omega) \leq \|x-z\| < \delta < \frac{\eps}{[w]_{Lip}+[\phi]_\B + 1} \qquad \forall\ z\in B_\delta(x)
\]
and $h_\eps(w-\phi) \equiv 0$ on $\partial\Omega^\eps$. In particular, $\xi_{\eps,\delta}:= h_\eps(w-\phi) * \eta_\delta$ vanishes on $\partial\Omega$.
As above, we see that
\[
\xi_{\eps,\delta} + \phi \to h_\eps(w-\phi) + \phi = \psi_\eps
\]
as $\delta\to 0$ and $\xi_{\eps,\delta}\in C_c^\infty(\R^d)\subseteq \B$. The proof concludes as in Step 1, using a diagonal sequence $\delta = \delta(\eps)$ such that $w_\eps = \phi + \xi_{\eps,\delta(\eps)} \to w$ in $W^{1,q}(\Omega)$ and uniformly.
\end{proof}

\begin{remark}\label{remark generalization}
Theorem \ref{theorem no barron-lipschitz gap} can be extended in multiple ways. 
\begin{enumerate}
\item A close inspection of the proof demonstrates that the proof does not require continuity of $f$ and $g$, but merely that
\begin{enumerate}
\item for any fixed $x\in\Omega$, the function $f = f(x, u, \xi)$ is lower semi-continuous in $(u, \xi)$. 
\item The function $f$ is jointly measurable in $(x,u,\xi)$.
\item For any compact set $K \subset \R\times \R^d$, there exists $\phi_K\in L^1(\Omega)$ such that
\[
\big|f(x, u, \xi)\big| \leq \phi_K(x) \qquad \forall\ u, \xi \in K
\]
and almost all $x\in\Omega$.
\end{enumerate}
One could similarly consider $f$ in the Caratheodory class, i.e.\ requiring measurability only in $x$ for fixed $u, \xi$, but strengthening the lower semi-continuity in $(u,\xi)$ to continuity.
The assumptions on $g$ can be weakened analogously.

\item One may consider vector-valued functions $U:\Omega\to \R^k$ and $f: \Omega \times \R^k \times \R^{k\times d}\to\R$ instead of the scalar case, defining vector-valued Barron functions as those for which all components are Barron. The same proof applies without modification. 

\item Any Lipschitz function is weakly differentiable and satisfies $\|\nabla u\|_{L^\infty} \leq [u]_{Lip}$, so the functionals in Theorem \ref{theorem no barron-lipschitz gap} are automatically defined on $C^{0,1}(\overline\Omega) \embeds W^{1,\infty}(\Omega)$.

For all connected and bounded Lipschitz-domains $\Omega$, the spaces $C^{0,1}(\overline\Omega)$ and $W^{1,\infty}(\Omega)$ coincide by Rademacher's Theorem \cite[Chapter 3.1]{evans2025measure}. If $\Omega$ is convex, then even $[u]_{Lip(\overline\Omega)} = \|\nabla u\|_{L^\infty(\Omega)}$. However, in the two-dimensional domain $\Omega:= B_1(0)\setminus \{x\geq 0, y=0\}$, functions like the principal branch of $(x,y) \mapsto Im\big((x+iy)^{3/2}\big)$ have bounded gradients, but do not extend to the closure of $\Omega$. 

Barron functions $u:\Omega\to\R$ on the other hand are defined by a representation formula and automatically extend to the entire space with $[u]_{Lip(\R^d)}\leq \|u\|_\B$, although not necessarily uniquely. If there is a gap between $W^{1,\infty}(\Omega)$ and $C^{0,1}(\overline\Omega)$, then the class of Barron functions is to be sorted with the smaller space $C^{0,1}(\overline\Omega)$.
\end{enumerate}
\end{remark}

\section{One-dimensional examples}\label{section 1d}

The first example of a Barron-Lipschitz gap we present is in the spirit of Mani\`a's celebrated one-dimensional example, showing that a gap may exist even in seemingly simple situations. In view of Theorem \ref{theorem no barron-lipschitz gap}, the functional is somewhat less innocuous than that of \cite{mania1934sopra} by necessity.

\begin{theorem}\label{theorem mania style example}
The functional
\[
J: W^{1,\infty}(0,1) \to [0,\infty], \qquad J(u) = \int_0^1 \frac{u^2}{x^3}\dx + \big\||u'|-1\big\|_{L^\infty(0,1)},
\]
satisfies
\[
\inf_{u\in W^{1,\infty}(0,1)}J(u) =0, \qquad \inf_{u\in \B}J(u) =1.
\]
\end{theorem}

\begin{proof}
{\bf Step 1. Lipschitz class.} First, we prove that $\inf_{u\in W^{1,\infty}(0,1)}J(u) =0.$ We define $u(x) = \sum f_j(x)$ where each $f_n$ is supported  on an interval
\[
I_n = [x_{n}, x_{{n-1}}]\qquad\text{of length $\ell_n$ such that}\quad x_n = 1- \sum_{j=1}^{n} \ell_j  = \sum_{j=n+1}^\infty \ell_j
\]
with a sequence $\{\ell_j\}_{j=1}^\infty$ to be determined later.  We denote the midpoint of the interval by $r_n = (x_n+x_{n-1})/2$ and define  $ f_n(x) = \max\{0, \ell_n/2-|x-r_n|\}$ meaning that $f_n=0$ outside of $I_n$.

As the intervals $I_n$ are disjoint and fill the entire interval $(0,1)$, the series $u(x) = \sum f_j(x)$ satisfies $|u'| =1$ at all points except for the midpoints $r_n$, boundary points $x_n$ of the intervals $I_n$, and the boundary point $0$ of $[0,1]$. As this set is countable, its measure is zero, so $u\in W^{1,\infty}(0,1)$ and  $ \big\||u'|-1\big\|_{L^\infty(0,1)} =0.$ Now we show that choosing intervals appropriately, for each $\eps$ we can construct $u$ such that $J(u) < \eps.$ We see that because for $x \in(0,1)$ we have $\frac{u^2}{x^3} \geq 0$ and $\cup_{k=1}^{\infty} I_j = (0,1)$ by the Monotone Convergence Theorem, $\int_{\cup_{k=1}^{n} I_j} \frac{u^2}{x^3}\dx  \xrightarrow{n \to \infty} \int_{0}^{1}  \frac{u^2}{x^3}\dx.$ As the intervals $\{I_j\}$ are disjoint we have:
\begin{align*}
J(u) = \lim_{n \to \infty} \int_{\cup_{k=1}^{n} I_j} \frac{u^2}{x^3}\dx = \lim_{n \to \infty} \sum_{j=1}^{n} \int_{I_j} \frac{u^2}{x^3}\dx = \sum_{j=1}^{\infty} \int_{I_j} \frac{{f_j}^2}{x^3}\dx,
\end{align*}
with the possibly infinite right hand side. 
We estimate:
	\begin{align}\label{fj}
		\sum_{j=1}^{\infty} \int_{I_j} \frac{{f_j}^2}{x^3}\dx &\leq \sum_{j=1}^{\infty} \int_{I_j} \frac{{\left(l_j/2\right)}^2}{x_n^3}\dx  
        = \frac14 \sum_{j=1}^{\infty} \left( \frac{{l_j}}{x_n} \right)^3 = \frac14 \sum_{n=1}^{\infty}\left(\frac{x_{n-1}-x_n}{x_n}\right)^3 = \frac14 \sum_{n=1}^{\infty} \left(\frac{x_{n-1}}{x_n} - 1\right)^3.
	\end{align}
Now by taking  $x_n = \frac{1}{\big(n+1\big)^\alpha},$ by \eqref{fj} we have:
	\begin{align*}
	J(u_\alpha) \leq \frac14	 \sum_{n=1}^{\infty}  \left( \left( 1 + \frac{1}{n}\right)^\alpha  -  1 \right)^3 \leq \frac14 \sum_{n=1}^{\infty} \frac{\alpha^3}{n^3} = \frac14 \alpha^3 \zeta(3) \leq \frac{3}{8}\alpha^3,
	\end{align*}
	where $\zeta$ is the Riemann $\zeta$-function.
	 Now, by sending $\alpha \to 0,$  $\lim_{\alpha \to 0} J(u_{\alpha}) = 0. $

{\bf Step 2. Barron class.} Now, we will show that $\inf_{u\in \B}J(u) =1.$ Since $u\in \B$ we have $u'\in BV(0,1).$ By \cite[Theorem 4.51]{apostol1958mathematical}, the one-sided limit $\lim_{t\to x^+} u'(t)$ exists for all $x \in [0,1)$.
Now we claim that if $u\in\B$ and $J(u) <\infty$, then $\lim_{x \to 0^+} u'(x) = u(0) = 0$, otherwise $J(u) = \infty$. Barron functions are in particular Lipschitz-continuous, so the point value $u(0)$ is well-defined. If $u(0) \neq 0$, then there exist $\eps, \delta>0$ such that $|u(x)|>\eps$ for all $x\in [0,\delta)$, so
\[
J(u) \geq \int_0^1 \frac{u^2}{x^3}\dx \geq \int_0^\delta \frac{\eps^2}{x^3}\dx = +\infty,
\]
leading to a contradiction.
Similarly, if $\alpha:= \lim_{x \to 0^+} u'(x) \neq 0 $, we note that there exists $\delta>0$ such that $|u'(x) - \alpha| < |\alpha|/2$ for all $x\in (0,\delta)$. In particular $u'$ does not change sign, so
\[
|u(x)| = \left|u(0) + \int_0^x u'(t)\dt\right| = \int_0^x \big|u'(t)\big|\dt \geq \frac{|\alpha|} 2\,x\qquad \forall\ x\in (0,\delta)
\]
since $u\in W^{1,\infty}(0,1)$ is absolutely continuous. In particular, we have
	 \begin{align*}
	 	\int_{0}^{1} \frac{u^2}{x^3} \dx \geq \int_{0}^{\delta} \frac{(\alpha/2)^2 \,x^2}{x^3}  \dx = \frac{\alpha^2}4  \int_{0}^{\delta} \frac{1}{x}\dx =  \infty.
	 \end{align*}
Again, we reach a contradiction. Consequently,  we have shown that $\lim_{x \to 0^+} u'(x) = 0. $ We note that $J(u) \geq \big\||u'|-1\big\|_{L^\infty(0,1)},$ and as $\lim_{x \to 0^+} u'(x) = 0$ we conclude that  $J(u) \geq 1.$ On the other hand, the constant function $u \equiv 0 \in \B$ satisfies $J(u) = 1$, so $\inf_{u\in \B}J(u) =1$.
\end{proof}
	
A similar gap  arises for functionals like
\[
J_{p,\alpha}(u) = \int_0^1 \frac{|u|^p}{x^\alpha}\dx + \big\||u'|-1\big\|_{L^\infty(0,1)}\qquad\text{and}\quad \widetilde J(u) = \left|\int_0^1 \frac{u'(x)}x\dx\right|+ \big\||u'|-1\big\|_{L^\infty(0,1)}
\]
with $p-\alpha<-1$, depending on whether the derivative can switch between $-1$ and $1$ a finite or infinite number of times. However, at least for $J$ and $J_{p,\alpha}$, minimizing sequences in both Lipschitz space and Barron space converge to the zero function uniformly and thus resemble each other -- it is merely the energy that is poorly approximated. We present a second example in which this is not the case. As the proof resembles that of the previous theorem, we only present a sketch.

\begin{theorem}\label{theorem second mania style example}
The functional
\[
J: W^{1,4}(0,1) \to [0,\infty], \qquad J(u) = \int_0^1 \frac{u^2(u-x)^2}{x^5} + \frac{\big((u')^2-1\big)^2}x + u^2 \dx
\]
satisfies
\[
\inf_{u\in W^{1,\infty}(0,1)}J(u) = \inf_{u\in W^{1,4}(0,1)}J(u)  =0, \qquad \inf_{u\in \B}J(u) >0.
\]
\end{theorem}

\begin{proof}[Sketch of Proof]
The Lipschitz case is analogous to the previous proof with minor modifications. To see this, note that
\[
\int_0^1 u^2\dx \leq \int_0^1 \frac{u^2}{x^3}\dx, \qquad \qquad\int_0^1 \frac{u^2(u-x)^2}{x^5}\dx = \int_0^1 \frac{u^2}{x^3}\,\left(\frac{x-u}x\right)^2\dx\leq \int_0^1\frac{u^2}{x^3}\dx
\]
if $0\leq u\leq x$, in particular for the functions constructed in the proof of Theorem \ref{theorem mania style example} which satisfy $u\geq 0$ by construction and $u\leq x$ since they are $1$-Lipschitz and take the value $u(0)=0$. As previously, the term involving $u'$ vanishes for this sequence.

For Barron functions, the limit $\alpha := \lim_{x\to 0^+} u'(x)$ exists as before. The second integrand forces that $\alpha\in\{-1,1\}$, and in combination with the first integrand, we find that $\alpha =1$ and $u(0) = 0$.

Define the `upper potential well set' $I^+ = \{ x\in (0,1) : 2x/3 < u(x)\}$, the `lower potential well set' $I^- = \{x\in (0,1) : u(x) < x/3\}$ and the `transition set' $I^t = \{x\in (0,1) : x/3 < u(x) < 2x/3\}$. Since $\alpha =1$, we see that $(0,\eps)\subseteq I^+$ for some $\eps>0$. We consider two cases.

{\bf Case 1: $I^-= \emptyset$.} In this case
\[
J(u) \geq \int_0^1 u^2 \dx  \geq \int_0^1 \left(\frac x3\right)^2\dx = \frac1{27}.
\]

{\bf Case 2:} $I^- \neq \emptyset$. Since $u$ is continuous, there exists (at least) one interval $[r,R]\subseteq I^t \subseteq [0,1]$ in which $u$ switches from $I^+$ to $I^-$. By continuity $u(r) = 2r/3$ and $u(R) = R/3$ and thus
\[
\int_0^1 \frac{u^2(u-x)^2}{x^5}\dx \geq \int_r^R \frac{(x/3)^2 (2x/3)^2}{x^5}\dx = \frac4{81}\int_r^R\frac1x = \frac4{81}\,\log\left(\frac Rr\right)
\]
since $u^2(u-x)^2$ is smallest at the boundary points of $[x/3, 2x/3]$. It remains to exclude the possibility of $r,R$ being too close. Note that the average slope in the interval must be 
\[
\frac1{R-r} \int_r^R u'\dx = \frac{u(R)-u(r)}{R-r} = \frac{R/3 - 2r/3}{R-r} = \frac13\left(\frac{R-r}{R-r} - \frac{r}{R-r}\right) = \frac13\left(1- \frac{1}{R/r-1}\right).
\]
In particular, 
\[
\frac Rr \leq \frac 54 \qquad \Ra \quad \frac1{R-r} \int_r^R u'\dx \leq \frac13 \left(1- \frac1{1/4}\right) = -1.
\]
By Jensen's inequality for the convex function $\phi(s) = \max\{s^2-1,0\}^2$, we find that
\begin{align*}
\int_0^1\frac{\big((u')^2-1\big)^2}x\dx &\geq \frac1R\int_r^R \max\{(u')^2-1, 0\}^2\dx\\
    &= \frac{R-r}R\,\frac1{R-r}\int_r^R \max\{(u')^2-1, 0\}^2\dx\\
    &\geq \frac{R-r}R\,\max \left\{ \left(\frac1{R-r}\int_r^Ru'\dx\right)^2-1, 0\right\}^2.
\end{align*}
In the regime $R/r \leq \frac 54$, we can replace $\phi(s)$ by $\tilde\phi(s) = 4(s+1)^2$ since $\phi''(s) = 12s^2 -4\geq 8$ for $s\in (-\infty, -1]$, leading to the simpler estimate
\begin{align*}
\int_0^1\frac{\big((u')^2-1\big)^2}x\dx 
    &\geq 4\frac{R-r}R \left(\frac13 \left(1 - \frac1{R/r-1}\right) + 1\right)^2
    = \frac49\left(1- \frac1{R/r}\right)\left(4- \frac1{R/r-1}\right)^2.
\end{align*}
In particular, either $R/r\geq 5/4$ and $J(u) \geq \frac4{81}\,\log(5/4)$ or $R/r\leq 5/4$ and
\[
J(u) \geq \inf_{s\in(1,\,5/4]} \frac{4}{81}\,\log(s) + \frac{4(s-1)}{9s}\left(4 - \frac1{s-1}\right)^2.
\]
The right term grows to infinity as $(s-1)^{-1}$ as $s\searrow 1$, so the infimum is strictly positive. In particular, in either case, there is no sequence of Barron functions realizing zero energy.
\end{proof}

Following the proof of the Theorem, a minimizing sequence in $W^{1,\infty}(0,1)$ should resemble saw-tooth functions of high frequency, low amplitude, and slope $\pm1$ with higher degrees of refinement at the origin. On the other hand, Barron functions of finite energy resemble the linear function $u(x) = x$ at the origin. So, it is not merely the energy that is poorly approximated, but in fact the shape of functions of almost minimal energy varies significantly between the two function classes. Notably, this limitation concerns not just neural networks, but any method of numerical function approximation relying on functions with derivatives in $BV(0,1)$ and thus many classical approaches.

\section{Folding and bending for graph surfaces}

\subsection{Model and Main Result}

Our next example is a simple model for stretching, bending and folding of thin shells. In a linearized model, we consider surfaces as graphs of functions $u:\Omega\to\R$ which are $H^2$-regular, possibly outside a small set $K$ where the surface `folds'. More formally, we denote by $\mathcal K_\Omega$ the set of relatively closed sets in $\Omega$ of finite $1$-dimensional Hausdorff measure and consider the class
\begin{equation}\label{eq bending surface class}
\mathcal G(\Omega) := \{(u,K) \in H^1(\Omega) \times \mathcal K_\Omega: \nabla u \in SBV^2(\Omega; \R^d ) \cap H^1(\Omega\setminus K;\R^d)\}
\end{equation}
where $SBV^2$ denotes the class of special functions of bounded variation whose regular part is square-integrable, see e.g.
\cite{ambrosio2000functions, bouchitte2002global, conti2017deformation, lazzaroni, ginster2026anisotropic}. We note that all functions considered explicitly below satisfy the stronger regularity assumption $u\in W^{1,\infty}(\Omega)$, strengthening the `gap' result between $\B$ and $\G$.

The energy we consider is a linearized toy model for geometric elastic energies derived from three-dimensional elasticity by means of $\Gamma$-convergence. The functional is composed of three components:
\begin{enumerate}
\item The out-of-plane bending contribution $\int_{\Omega\setminus K} \|D^2 u\|^2\dx$,
\item the in-plane stretch (or `isometric embedding') contribution proportional to $\int_\Omega \|\nabla u \otimes \nabla u - g\|^2\dx$ where $g:\Omega \to \R^{d\times d}_{spd}$ is a target (Riemannian) metric and 
\item the folding contribution, proportional to the measure $\H^1(K)$.
\end{enumerate}
All matrix norms are Frobenius norms.
This functional is both geometrically linear (linearized around the identity deformation) and additionally linearized for graph surfaces. Similar ideas can be found for instance in \cite{bonito2023finite}. Overall, the functional takes the form\footnote{\ Despite the name, our parameters $\lambda,\mu$ are unrelated to the Lam\'e moduli of materials and should be thought of as general parameters for now.}\ $E_{\mu,\lambda, g}: \mathcal G(\Omega) \to [0,\infty]$ with
\begin{equation}\label{eq bending energy}
E_{\mu,\lambda, g}(u,K) = \begin{cases}\int_{\Omega \setminus K} \|D^2u\|^2\dx + \mu \int_\Omega \big\|\nabla u\otimes \nabla u - g\big\|^2 \dx +  \lambda \,\H^1(K) & \text{if }u\in H_0^1(\Omega)\\ +\infty &\text{else.}
\end{cases}
\end{equation}
The energy scales merely with the size of the bending set, while the folding angle $|D^2u|_{sing}$ does not enter, suggesting high plastic flexibility beyond the elastic regime. 
We enforce anchored boundary conditions for $u$, but the result also applies under the clamped boundary condition $\equiv 0$, $\partial_\nu u=1$ on $\partial\Omega$.

The bending component of the elastic energy of a sheet of thickness $h\ll 1$ scales with $h^3$ while the stretching component scales with $h$ when derived from three-dimensional elasticity theory (see e.g.\ \cite{friesecke2006hierarchy} and the sources cited there). In our notation, this suggests that $\mu \sim h^{-2} \gg 1$. The folding parameter $\lambda$ has to be of moderate size.

To exploit radial symmetry, we focus on an annulus $\Omega = B_R\setminus B_r$ in two dimensions where $r, R$ have comparable magnitude. For this, we use the polar basis vectors
\[
e_r = \big(\cos\phi, \sin\phi), \qquad e_\phi = (-\sin \phi, \cos\phi)
\]
and recall that the Hessian takes the form
\begin{equation}\label{eq hessian polar}
u_{rr}\, e_r\otimes e_r +\frac{u_{r\phi}  - u_\phi}r\big(e_r\otimes e_\phi + e_\phi \otimes e_r\big) + \frac{u_{\phi\phi} + r\,u_r}{r^2}\,e_\phi \otimes e_\phi
\end{equation}
in polar coordinates. Our main result is the following. 

\begin{theorem}\label{theorem bending}
Let $\Omega$ be the annular domain $\{x\in\R^2 : r<\|x\|<R\}$ for $0<r<R$ and $E_{\mu,\lambda, g}$ as in \eqref{eq bending energy} with $g = e_r \otimes e_r$. Then, if 
\[
0 < r < R < \exp(1/7)\,r, \qquad
\frac{2R\left(1+\log\left(\frac{R}{r}\right)\right)}
{4r^2-rR-R^2} \leq \lambda < \frac92\,\frac{1/7 - \log(R/r)}{(1-r/R)\,(2r+R)} . 
\]
and we denote by $e_{min} = 1+ \log(R/r) + \lambda\,(R+r)/2$ and assume that
\begin{align*}
 \mu &\geq \max\Biggl\{ \frac{16\,e_{min}}{9r(R-r)},\ \left[ \frac{12}{11r} \,e_{min}
  \right]^2,\
  \frac{30}{r(R-r)} \left( e_{min}-\lambda\,\frac{r^2}{R} \right),\
\left[ \frac{10}{r} \left(e_{min}-\lambda\,\frac{r^2}{R} \right) \right]^2, \\ &\hspace{1 cm}\frac{3}{8r^2\lambda(2r+R)}, \ \frac{e_{min}-1}{2r^2\left[\frac 17 - \log\left(R/r\right) - 2\lambda (1-r/R) (2r+R)/9\right]^2} \ \Biggl\},
\end{align*}
then there exists a composition of two Barron functions $u_1:\R\to\R,\ u_2:\R^2\to\R$ and a circle $S$ such that
\[
\inf_{(u,K) \in\mathcal G(\Omega)} E(u,K) + \frac{2\pi}{49} \leq E(u_1\circ u_2, S) + \frac{2\pi}{49} < E(\tilde u, \tilde K)
\]
for every $\tilde u \in \B$ and any closed set $\tilde K$ such that $\tilde u\in H^2(\Omega \setminus \tilde K)$. 
\end{theorem}

Much like the composition of two neural networks with $L_1$ and $L_2$ non-linear activations respectively is a neural network with $L_1+L_2$ non-linear activation functions, the composition of two Barron functions lies in both a tree-like function space \cite{wojtowytsch2020banach} (designed to capture wide feed-forward ReLU networks) and the non-linear flow-induced function class \cite{ma2022barron} (designed to describe deep residual networks). Thus, the representational obstacles in this situation can be overcome by passing to deeper neural networks, establishing a `depth separation' result.

For instance, the theorem applies with moderate values $r = 14,\ R = 15,\ \lambda = 0.1$ and $\mu = 40$.
We estimate in the proof that 
\[
\inf_\G E(u,K) \leq 2\pi\left( \lambda\,\frac{r+R}2 + \log\left(\frac Rr\right)  \right) \approx 2\pi \cdot 1.5,
\]
meaning that the relative energy gap is roughly $1.3\%$. Under the scaling $(R,r)\mapsto (\gamma R, \gamma r)$, the admissible parameters scale as $(\lambda,\mu)\mapsto (\lambda/\gamma, \mu/\gamma^2)$ since we prescribed an absolute energy gap. The energy is scale-invariant under the simultaneous dilation $K_\gamma = \gamma K$ and scaling $u_\gamma(x) = \gamma\,u(x/\gamma)$.

While the proof of Theorem \ref{theorem bending} exploits radial symmetry, it is resilient to small perturbations and we expect that a gap $inf_{(u,K)\in \mathcal G(\Omega)} E_{\mu,\lambda, g}(u,K) +\eps\leq   E_{\mu,\lambda,g}(\tilde u,\tilde K)$ with $\tilde u\in\B$ persists on `rough annulus' domains $\Omega = \{ (s\cos\theta, s\sin\theta) : g_1(\theta)<s<g_2(\theta)\}$ if $\|g_1 - r\|_{C^2},\|g_2-R\|_{C^2}$ are sufficiently small and $r,R,\mu,\lambda$ satisfy the conditions of Theorem \ref{theorem bending}. However, the construction of an explicit energy competitor becomes more subtle and we do not claim the same `depth separation' result.

The proof is based on the following `slicing lemma' where the functional $E_{\mu, \lambda,g}$ is `localized' along rays $R_v:= \{ tv/\|v\| : t\in (r,R)\}$. We use two energies: $\widetilde F$ with a guaranteed lower bound for the norm of the full Hessian for all functions and $F$ with the exact expression for radially symmetric functions.

\begin{lemma}\label{lemma slicing}
Let $r<R$. Denote by $\mathcal E$ the collection of all finite subsets of $(r,R)$ and set
\[
\mathcal G_1 = \big\{(u, K) \in H^1_0(r, R) \times \mathcal E : u\in H^2\big((r,R)\setminus K\big)\big\}.
\]
For $\mu, \lambda>0$ define $F, \widetilde F: \mathcal G_1 \to [0,\infty]$
\begin{align*}
\widetilde F(u, K) &= \int_{(r,R)\setminus K} |u''|^2\,s \ds + \mu\int_r^R \big((u')^2-1\big)^2\,s\ds + \lambda \sum_{s\in K}\frac{s^2}R\\
F(u, K) &= \int_{(r,R)\setminus K} \left(|u''|^2 + \left(\frac{u'}{s}\right)^2\right)\,s \ds + \mu\int_r^R \big((u')^2-1\big)^2\,s\ds + \lambda \sum_{s\in K}s.
\end{align*}
Then, if $r, R, \lambda,\mu$ are as in Theorem \ref{theorem bending}, then
\[
\widetilde F(u, K) \geq \inf_{(v, L) \in \G_1} F(v, L) + 1
\]
if $\left( u,K \right)\in \G_1$ and $K$ is either
\begin{enumerate}
    \item empty, 
    \item contains at least two distinct points, or
    \item $K$ is {\em not} contained in the middle third of the interval: $K \not\subseteq (2r/3 + R/3, r/3 + 2R/3)$.
\end{enumerate}
For any $(u,K)$, we have
\[
\widetilde F(u, K) \geq \inf_{(v,L)\in\mathcal G_1}F(v, L) - \frac1{7}.
\]
Both $F$ and $\widetilde F$ have minimizers $(u^*, K^*)$ and $(\tilde u^*, \tilde K^*)$ in $\mathcal G_1$.
\end{lemma}

Put in physical terms, a single fold in the middle of the interval is energetically favorable, but additional folding or folding towards the edges of the interval is not. In the original two-dimensional problem, this corresponds to the observation that a single circular fold close to the mid circle of the annulus is energetically favorable, but folds along straight lines do not help reduce the elastic energy. As Barron functions can only describe folding along entire straight lines by Theorem \ref{theorem structure barron sbv2}, an energy gap arises.

The restrictions on $R/r$ arising in Lemma \ref{lemma slicing} (and by proxy, Theorem \ref{theorem bending}) could likely be relaxed.
There is another key ingredient to the proof of Theorem \ref{theorem bending} by slicing. Specifically, we use the following auxiliary lemmas to ensure that there are enough `bad' slices along which the energy is high if the singular set of a function is a union of straight lines (e.g.\ if $u\in \B \cap SBV^2$ by Theorem \ref{theorem structure barron sbv2}). Initially, we show this for a single straight line and subsequently for a union of lines.

\begin{lemma}\label{lemma angles}
Let $R, r>0$ and $L$ be a line in $\R^2$ which does not pass through the origin but intersects $B_R(0)$. Define $\pi: \R^2\setminus\{0\}\to S^1$ by $\pi(x) = x/|x|$ and
\[
A = \pi\big(L\cap B_R\setminus B_b\big), \qquad B = \pi\big(L\cap B_b\setminus B_a\big), \qquad C = \pi\big(L\cap B_a\setminus B_r\big)
\]
for $a = \frac{2r+R}3$ and $b= \frac{r+2R}3$. Then 
\[
\frac{|B|}{|A| + |C|} \leq \frac{\cos^{-1}(a/b)}{\cos^{-1}(a/R) - \cos^{-1}(a/b)} \leq \frac{\pi/3}{\cos^{-1}(1/3) - \pi/3} < 5.7 \leq 6
\]
where $|\cdot|$ denotes the (one-dimensional Hausdorff) measure of a set on the circle. If $R/r \leq 3/2$, the stronger bound $|B| \leq 3\,|A|$ holds.
\end{lemma}

\begin{figure}
\begin{tikzpicture}[scale = 2.4]
\filldraw[fill = red!30](2,0) arc (0:180:2);
\filldraw[fill = blue!30](5/3,0) arc (0:180:5/3);
\filldraw[fill = orange!30](4/3,0) arc (0:180:4/3);
\filldraw[fill = white](1,0) arc (0:180:1);
\node[below] at(1,0){$r$};
\node[below] at(4/3,0){$\frac{2r+R}3$};
\node[below] at(5/3,0){$\frac{r+2R}3$};
\node[below] at(2,0){$R$};

\fill[very thick, fill = blue!50, domain = 43:60] (0,0) -- plot ({.5*cos(\x)}, {.5*sin(\x)});
\fill[very thick, fill = red!50, domain = 35:43] (0,0) -- plot ({.5*cos(\x)}, {.5*sin(\x)});
\fill[very thick, fill = blue!50, domain = 43:60] (0,0) -- plot ({-.5*cos(\x)}, {.5*sin(\x)});
\fill[very thick, fill = red!50, domain = 35:43] (0,0) -- plot ({-.5*cos(\x)}, {.5*sin(\x)});
\fill[fill = orange!50, domain = 60:120] (0,0) -- plot ({-.5*cos(\x)}, {.5*sin(\x)});

\draw[very thick](-2.1,1.15) -- ++ (4.2,0);
\draw[thick](0,0) -- ({sqrt(4 - 1.15^2}, 1.15);
\draw[thick](0,0) -- ({-sqrt(4 - 1.15^2}, 1.15);
\draw[thick](0,0) -- ({sqrt((5/3)^2 - 1.15^2}, 1.15);
\draw[thick](0,0) -- ({-sqrt((5/3)^2 - 1.15^2}, 1.15);
\draw[thick](0,0) -- ({sqrt((4/3)^2 - 1.15^2}, 1.15);
\draw[thick](0,0) -- ({-sqrt((4/3)^2 - 1.15^2}, 1.15);

\node[above, orange] at (0, .5) {$C$};
\node[above, red] at (.4, .05) {$A$};
\node[above, red] at (-.4, .05) {$A$};
\node[blue] at (.37, .47) {$B$};
\node[blue] at (-.37, .47) {$B$};
\end{tikzpicture}

\includegraphics[height = 3.5cm]{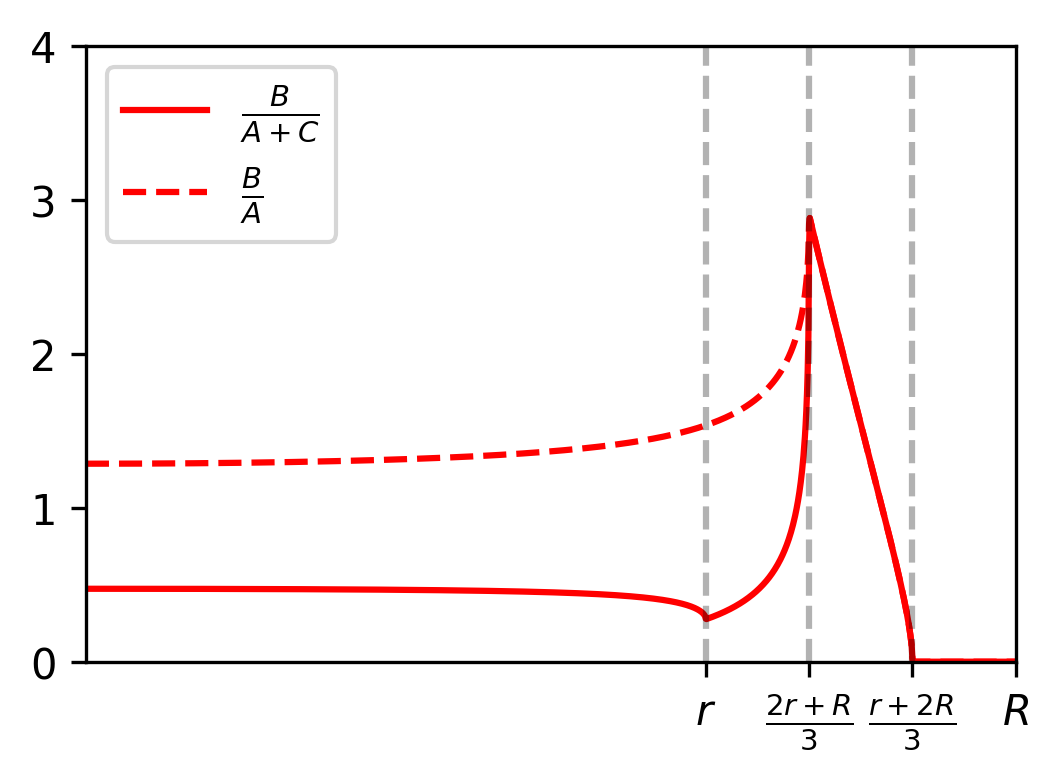}
\includegraphics[height = 3.5cm]{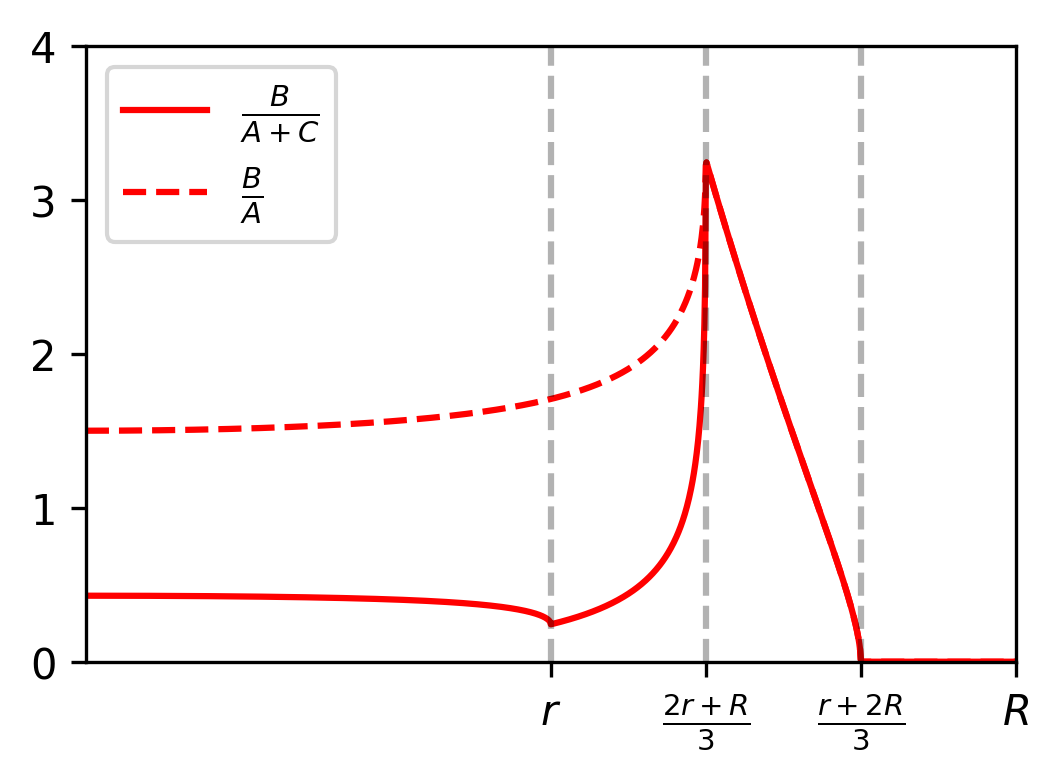}
\includegraphics[height = 3.5cm]{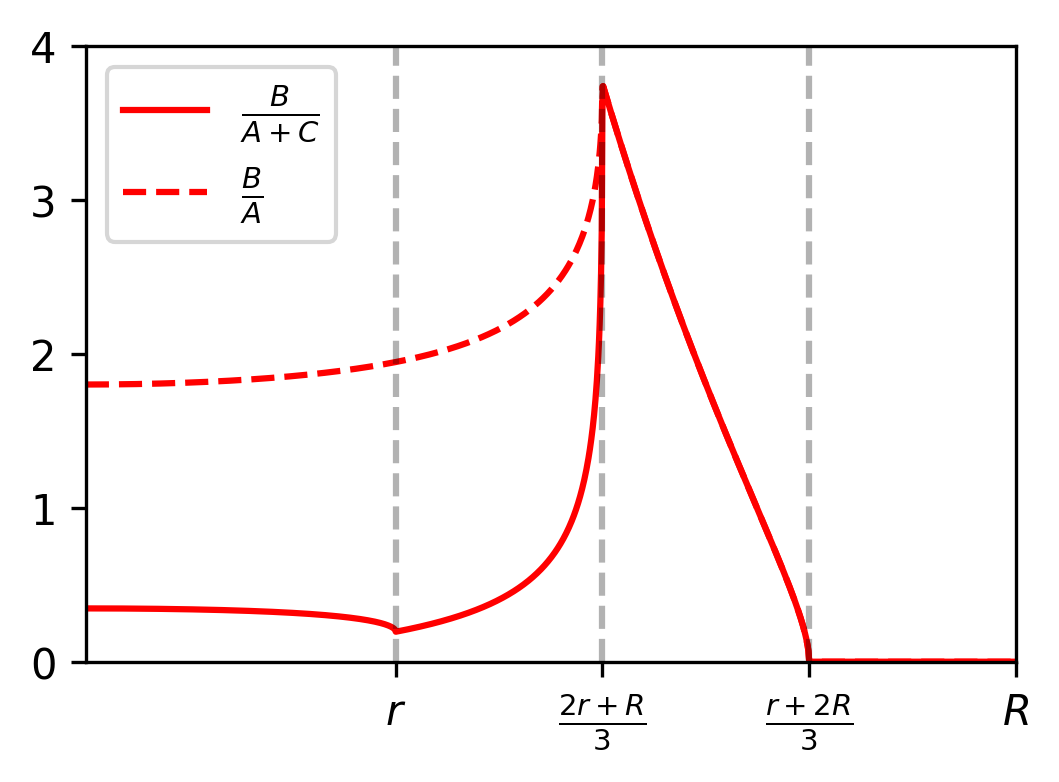}

\caption{\label{figure angles}
{\bf Top row:} The sets $A, B, C$ on the circle correspond to the angles on the unit circle for which the line intersects the inner, central and outer ring respectively. Depending on the distance of the line to the origin one (or multiple) of the sets may be empty, and they may have one or two connected components, see also Figure \ref{figure angles situations}. {\bf Bottom row:} Quotients of the measure of the regions $B/A$ and $B/(A+C)$ as a function of the distance of the line to the origin. If the distance is greater than $R$, then the line does not intersect any of the three circles. In all three plots we choose $r=1$ and vary $R= 1.5$ (left), $R=2$ (middle) and $R=3$ (right).}
\end{figure}
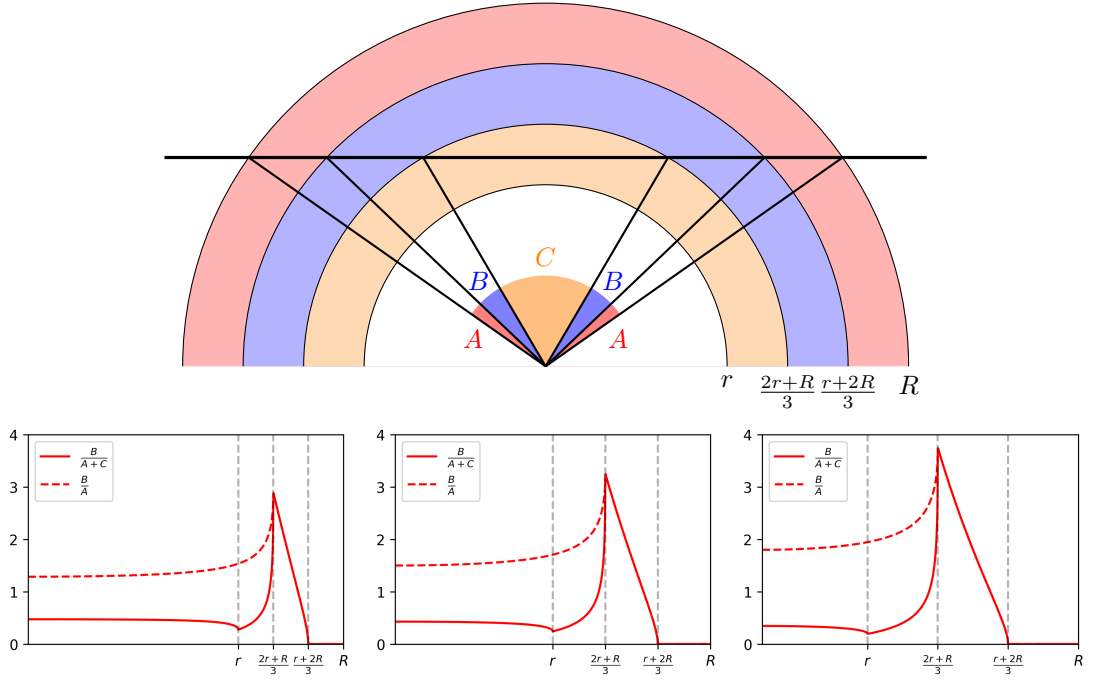

Specifically, the sets $A, C$ correspond to `bad' slices where the line passes through the inner or outer third of the annulus rather than its mid segment. See Figure \ref{figure angles} for a graphical explanation.
We extend these notions to unions of affine lines taking into account the entire circle, not just the segment along which a given line intersects it. For this, we introduce new notation: Given an annulus $B_R\setminus B_r$ and intermediate radii $a = \frac{2r+R}3$ and $b = \frac{r+2R}3$, for a collection of straight lines $\mathcal L =\{\ell_1,\dots, \ell_n\}$, we define the set of `good angles' as
\[
G(\mathcal L) = \left\{\nu \in S^1 : \langle\nu\rangle_+ \cap (B_R\setminus B_r) \cap \left(\bigcup_{i=1}^n \ell_i\right) \text{ is a singleton inside the sub-annulus }B_b\setminus B_a\right\}
\]
and the set of `bad angles' as $B(\mathcal L) = S^1\setminus G(\mathcal L)$. By $\langle \nu\rangle_+$ we mean the ray $\{t\nu : t\geq 0\}$.

\begin{lemma}\label{lemma multiple lines}
Let $\mathcal L$ be a countable and locally finite collection of lines and $B_R\setminus B_r$ an annulus. Then the set of bad angles $B(\mathcal L)$ has measure at least $2\pi/13$. If $R/r \leq 3/2$, then $B(\mathcal L)$ has measure at least $2\pi/7$.
\end{lemma}

A visual representation of Lemma \ref{lemma multiple lines} is given in Figure \ref{figure bad angles}. In the remainder of this section, we give the proofs of first the auxiliary Lemmas and finally Theorem \ref{theorem bending}.

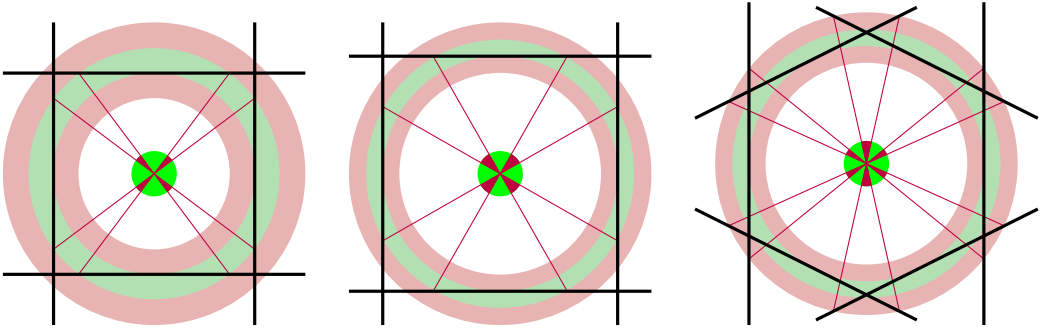
\begin{figure}
    \centering

\begin{tikzpicture}
\fill[color = darkred!30](0,0) circle (2cm);
\fill[color = darkgreen!30](0,0) circle (1.66cm);
\fill[color = darkred!30](0,0) circle (1.33cm);
\fill[color = white!30](0,0) circle (1.0cm);

\fill[color = green](0,0) circle (.3cm);

\draw[purple](0,0) -- (4/3, 1);
\draw[purple](0,0) -- (1, 4/3);
\fill[purple, domain = 36:52, samples =30] (0,0) -- plot ({.3*cos(\x)}, {.3*sin(\x)}) -- (0,0);

\begin{scope}[rotate =90]
\draw[purple](0,0) -- (4/3, 1);
\draw[purple](0,0) -- (1, 4/3);
\fill[purple, domain = 36:52, samples =30] (0,0) -- plot ({.3*cos(\x)}, {.3*sin(\x)}) -- (0,0);
\end{scope}

\begin{scope}[rotate = 180]
\draw[purple](0,0) -- (4/3, 1);
\draw[purple](0,0) -- (1, 4/3);
\fill[purple, domain = 36:52, samples =30] (0,0) -- plot ({.3*cos(\x)}, {.3*sin(\x)}) -- (0,0);
\end{scope}

\begin{scope}[rotate = 270]
\draw[purple](0,0) -- (4/3, 1);
\draw[purple](0,0) -- (1, 4/3);
\fill[purple, domain = 36:52, samples =30] (0,0) -- plot ({.3*cos(\x)}, {.3*sin(\x)}) -- (0,0);
\end{scope}

\draw[very thick] ( -2, 1.33) -- ++ (4,0);
\draw[very thick] ( -2,-1.33) -- ++ (4,0);
\draw[very thick] (1.33, -2) -- ++ (0,4);
\draw[very thick] (-1.33, -2) -- ++ (0,4);
\end{tikzpicture}
\hspace{3mm}
\begin{tikzpicture}[scale = 2/3]
\fill[color = darkred!30](0,0) circle (3cm);
\fill[color = darkgreen!30](0,0) circle (2.66cm);
\fill[color = darkred!30](0,0) circle (2.33cm);
\fill[color = white!30](0,0) circle (2.0cm);

\fill[color = green](0,0) circle (.45cm);

\draw[purple](0,0) -- (2.33, 1.33);
\draw[purple](0,0) -- (1.33, 2.33);
\fill[purple, domain = 30:60, samples =30] (0,0) -- plot ({.45*cos(\x)}, {.45*sin(\x)}) -- (0,0);

\begin{scope}[rotate =90]
\draw[purple](0,0) -- (2.33, 1.33);
\draw[purple](0,0) -- (1.33, 2.33);
\fill[purple, domain = 30:60, samples =30] (0,0) -- plot ({.45*cos(\x)}, {.45*sin(\x)}) -- (0,0);
\end{scope}

\begin{scope}[rotate = 180]
\draw[purple](0,0) -- (2.33, 1.33);
\draw[purple](0,0) -- (1.33, 2.33);
\fill[purple, domain = 30:60, samples =30] (0,0) -- plot ({.45*cos(\x)}, {.45*sin(\x)}) -- (0,0);
\end{scope}

\begin{scope}[rotate = 270]
\draw[purple](0,0) -- (2.33, 1.33);
\draw[purple](0,0) -- (1.33, 2.33);
\fill[purple, domain = 30:60, samples =30] (0,0) -- plot ({.45*cos(\x)}, {.45*sin(\x)}) -- (0,0);
\end{scope}

\draw[very thick] ( -3, 2.33) -- ++ (6,0);
\draw[very thick] ( -3, -2.33) -- ++ (6,0);
\draw[very thick] (2.33, -3) -- ++ (0,6);
\draw[very thick] (-2.33, -3) -- ++ (0,6);

\end{tikzpicture}
\hspace{3mm}
\begin{tikzpicture}[scale = 2/3]
\fill[color = darkred!30](0,0) circle (3cm);
\fill[color = darkgreen!30](0,0) circle (2.66cm);
\fill[color = darkred!30](0,0) circle (2.33cm);
\fill[color = white!30](0,0) circle (2.0cm);

\fill[color = green](0,0) circle (.45cm);

\draw[purple](0,0) -- (2.33, {sqrt(9 - 2.33^2});
\draw[purple](0,0) -- (2.735, {sqrt(9 - 2.735^2)} );
\fill[purple, domain = 25:40, samples =30] (0,0) -- plot ({.45*cos(\x)}, {.45*sin(\x)}) -- (0,0);

\begin{scope}[xscale = -1]
\draw[purple](0,0) -- (2.33, {sqrt(9 - 2.33^2});
\draw[purple](0,0) -- (2.735, {sqrt(9 - 2.735^2)} );
\fill[purple, domain = 25:40, samples =30] (0,0) -- plot ({.45*cos(\x)}, {.45*sin(\x)}) -- (0,0);
\end{scope}

\begin{scope}[yscale = -1]
\draw[purple](0,0) -- (2.33, {sqrt(9 - 2.33^2});
\draw[purple](0,0) -- (2.735, {sqrt(9 - 2.735^2)} );
\fill[purple, domain = 25:40, samples =30] (0,0) -- plot ({.45*cos(\x)}, {.45*sin(\x)}) -- (0,0);
\end{scope}

\begin{scope}[xscale = -1, yscale = -1]
\draw[purple](0,0) -- (2.33, {sqrt(9 - 2.33^2});
\draw[purple](0,0) -- (2.735, {sqrt(9 - 2.735^2)} );
\fill[purple, domain = 25:40, samples =30] (0,0) -- plot ({.45*cos(\x)}, {.45*sin(\x)}) -- (0,0);
\end{scope}

\draw[purple](0,0) -- (-.655, {sqrt(9 - .655^2)} );
\draw[purple](0,0) -- (.655, {sqrt(9 - .655^2)} );
\fill[purple, domain = 76:104, samples =30] (0,0) -- plot ({.45*cos(\x)}, {.45*sin(\x)}) -- (0,0);

\begin{scope}[yscale = -1]
\draw[purple](0,0) -- (-.655, {sqrt(9 - .655^2)} );
\draw[purple](0,0) -- (.655, {sqrt(9 - .655^2)} );
\fill[purple, domain = 76:104, samples =30] (0,0) -- plot ({.45*cos(\x)}, {.45*sin(\x)}) -- (0,0);

\end{scope}

\draw[domain = -1: 3.4, samples = 2, very thick] plot({\x}, {.5*\x-2.6});
\draw[domain = -1: 3.4, samples = 2, very thick] plot({-\x}, {.5*\x-2.6});
\draw[domain = -1: 3.4, samples = 2, very thick] plot({\x}, {-.5*\x+2.6});
\draw[domain = -1: 3.4, samples = 2, very thick] plot({-\x}, {-.5*\x+2.6});
\draw[very thick] (-2.33, -3.2) -- ++ (0, 6.4);
\draw[very thick] (2.33, -3.2) -- ++ (0, 6.4);

\end{tikzpicture}
\caption{\label{figure bad angles}The set of `bad angles' (red) and `good angles' (green) for various line arrangements and different values of $r$ and $R$. The `good' set corresponds to the set $B$ in Figure \ref{figure angles}, with the caveat that angles for which two lines are intersected by the ray automatically count as `bad'.}
\end{figure}

\subsection{Auxiliary Lemmas: One-dimensional slices}
We finally prove Lemma \ref{lemma slicing}.

\begin{proof}[Proof of Lemma \ref{lemma slicing}]
{\bf Step 1.} To find an upper bound on $\inf F$, we construct a piecewise linear energy competitor with a single non-smooth point. We define \(u\)  and $K$  by
\[
u(x)= \dist(x, \{r,R\}) = \frac{R-r}{2}-\left||x|-\frac{R+r}{2}\right|, \qquad  K = \{(r+R)/2\}.
\]
Then  $u\in W^{1,\infty}(r,R)\subseteq H^1(r,R)$  and \(u\in H^2\big((r,R)\setminus K\big)\) as well as $u(r) = u(R) = 0$, hence \((u,K)\in \mathcal G_1\). Moreover,
\[
u'(x)=1 \quad \text{for } r<x<\frac{r+R}{2},
\qquad
u'(x)=-1 \quad \text{for } \frac{r+R}{2}<x<R,
\]
and $u''(x)=0$ for $x\neq (R+r)/2$. Therefore,
\[
\inf_{(v,L)\in \mathcal G_1}F(v,L)
\leq
F(u,K)
=
\int_r^R \frac{1}{s} \,ds+\lambda \frac{r+R}{2}
=
\log \left(\frac{R}{r}\right)+\lambda \frac{    r+R}{2}.
\]
While the Euclidean norm on $\R^d$ cannot be represented exactly by a neural network of finite size for $d\geq 2$, it is a Barron function in any dimension with
\[
g(x) = c_d\,\mathbb E_{\nu\sim\pi_0} \big[\sigma(\nu^Tx)\big], \qquad c_d =\|g\|_\B = \frac 2{\int_0^1 (1-s^2)^\frac{d-3}2\,s\ds} = 2(d-1)
\]
with $\sigma(s) = \max\{s,0\}$ by \cite{wojtowytsch2022optimal}. The energy competitor $u(x) = \frac{R-r}2 - \left|\|x\|_2 -\frac{R+r}2\right|$ can therefore be represented by the composition of Barron functions
\[
u_2(x) = \|x\|_2 - \frac{R+r}2, \qquad u_1(s) = \frac{R+r}2 - |s|
\]
since constant functions and the absolute value $|s| = \sigma(s) + \sigma(-s)$ are Barron. 

{\bf Step 2.} When $|K|\geq 2,$ we have
\begin{align*}
\widetilde F(u, K) &= \int_{(r,R)\setminus K} |u''|^2\,s \ds + \mu\int_r^R \big((u')^2-1\big)^2\,s\ds + \lambda \sum_{s\in K}\frac{s^2}{R}
\\&\geq\lambda \sum_{s\in K}\frac{s^2}{R} \geq \lambda \frac{r^2}{R} |K| \geq \frac{2\lambda r^2}{R}.
\end{align*}
For now, we collect bounds to be analyzed and compared in Step 5 of this proof.

{\bf Step 3.} Consider $K = \emptyset$, i.e.\ $u\in H^2(r,R)$. Since $u(r) = u(R)$ we have $\int_r^R u'(s)\ds  =0$. The analysis of this step follows that of the Modica-Mortola or Ginzburg-Landau energy functional \cite{alberti2000variational}. Consider two options:
\begin{enumerate}
\item $|u'(s)|\leq 1/2$ for all $s$. Then
\[
\widetilde F(u, K) \geq \mu\int_r^R\left(1-(u')^2\right)^2s\ds\geq \mu r\int_r^R \left(\frac34\right)^2  \ds = \frac{9}{16}\left(R-r\right)r \mu.
\]
\item Assume that there exists $s^*$ such that $|u'(s^*)| >\frac12$. Since the integral of $u'$ vanishes, there furthermore exists a point $s_*$ such that $u'$ takes opposite signs at $s^*$ and $s_*$. 
Without loss of generality, we simplify notation by assuming that $s_*<s^*$ and $u'(s_*) < 0 < 1/2 < u'(s^*)$. Then, by Young's inequality:
\[
\int_r^R \left\{(u'')^2 + \mu \big((u')^2-1\big)^2\right\}s\ds \geq 2\sqrt{\mu}r\int_r^R\big|1-(u')^2\big|\,|u''|\ds \geq 2\sqrt\mu r\int_{s_*}^{s^*} \big|1 - (u')^2\big|\,u''\ds.
\]
Denoting an anti-derivative of $z\mapsto |1-z^2|$ by $G$, we find that
\[
\widetilde F(u, K) \geq 2\sqrt\mu r\int_{s_*}^{s^*} \frac{d}{ds} G(u'(s))\ds \geq 2\sqrt\mu r\big(G(1/2) - G(0)\big) = \frac{11}{12}\sqrt\mu\,r.
\]
\end{enumerate}
So if $K=\emptyset$, we have
\[
\widetilde F(u,K) \geq \min\left\{\frac{9}{16}(R-r)r\mu, \: \frac{11}{12}\,\sqrt\mu\,r\right\}.
\]

{\bf Step 4.} Assume that $K = \{\bar s\}$ with $\bar s \in\left(r, \frac{2r+R}3\right)$. The argument in this step is a slightly more involved version of that in Step 3. 

Denote the average slopes on the subintervals by $a = \frac1{\bar s - r} \int_{r}^{\bar s} u'(s)\dx$ and $b = \frac1{R-\bar s} \int_{\bar s}^R u'(s)\ds$. Then
\[
0 = u(R) - u(r) = \int_r^R u'(s)\ds = \int_{r}^{\bar s} u'(s)\ds + \int_{\bar s}^Ru'(s)\ds = (\bar s-r) a + (R-\bar s) b
\]
so 
\[
a = -\frac{R-\bar s}{\bar s -r}\,b\qquad \Ra \quad |a| = \left|\frac{R-\bar s}{\bar s - r}\right| |b| \geq \frac{2(R- r)/3}{(R- r)/3}\,|b| = 2|b| \quad\text{since }\bar s\in \left(r, \frac{2r+R}3\right).
\]
In particular, $|a|$ and $|b|$ cannot both be close to $1$.
Assume first that $|b|\leq 2/3$. As before, we distinguish two cases:
\begin{itemize}
\item If $|u'| \leq 5/6$ for all $s \in (\bar s, R)$, then
\[
\widetilde F(u,K) - \frac{\lambda r^2}{R} \geq \mu \int_{\bar s}^R \big((u')^2-1\big)^2\,s \ds \geq \mu r \int_{\bar s}^R \left(\left(\frac56\right)^2 - 1\right)^2 \ds \geq \frac{\mu r(R-\bar s)}{20} \geq \frac{\mu r(R-r)}{30}
\]
since $((5/6)^2-1)^2 = 121/1296\geq 1/20$ and $R-\bar s \geq 2(R-r)/3$.
\item Else, there must exist $s_1, s_2$ such that $|u'(s_1)| \leq 2/3$ and $|u'(s_2)|\geq 5/6$ and as in the previous step:
\[
\widetilde F(u, K) \geq 2\sqrt\mu r \big(G(5/6) - G(2/3)\big)  +\lambda \frac{r^2}R\geq \frac{\sqrt\mu r}{10} + \frac{\lambda r^2}{R}
\]
since
\[
G(5/6) - G(2/3) = \int_{2/3}^{5/6} 1-z^2\dz = \frac{47}{648}\geq \frac1{20}.
\]
\end{itemize}
Finally, if $|b|\geq 2/3$, then $|a|\geq 4/3$ and a similar argument as above goes through, assuming that the interval $\bar s-r$ is not too short to exclude constant functions from the set of candidates of low energy. To exclude this possibility, we follow the argument more carefully. Note that
\[
|a| \geq \frac{R-\bar s}{\bar s-r}|b| \geq \frac23\,\frac{R-\bar s}{\bar s-r},
\]
so if $\bar s-r$ is small, the lower bound becomes tighter.
\begin{itemize}
\item Note that $\phi(z) = (1-z^2)^2$ is a non-negative convex function on $[1,\infty)$. Thus, if $u'\geq 1$ on $(r, \bar s)$, then by Jensen's inequality we have
\[
\int_{r}^{\bar s} \phi(u'(s))s\ds \geq r\frac{\bar s-r}{\bar s-r}\int_r^{\bar s}\phi(u'(s))\ds \geq r(\bar s-r)\,\phi\left(\frac1{\bar s-r}\int_{\bar s}^r u'(s)\ds\right) = r(\bar s-r)\,\phi(a).
\]
In particular,
\begin{align*}
\widetilde F(u,K) -\frac{\lambda r^2}{R}&\geq \mu r(\bar s-r)\left(\left(\frac23\,\frac{R-\bar s}{\bar s-r}\right)^2 - 1\right)^2 
    = \mu r(\bar s-r)\left(\frac{4(R-\bar s)^2 - 9 (\bar s-r)^2}{9 (r-\bar s)^2}\right)^2\\
    &\geq \mu r(\bar s-r)\left(\frac{4(R-\bar s)^2 - (9/4) (R-\bar s)^2}{9 (r-\bar s)^2}\right)^2
    = \mu r(\bar s-r)\left({\frac{7}{9 \cdot 4}}\,\frac{(R-\bar s)^2}{(\bar s-r)^2}\right)^2\\
    &= {\frac{49(R-\bar s)\mu r}{1296}} \left(\frac{R-\bar s}{\bar s-r}\right)^3  \geq \frac{98(R-r)\mu r}{3\cdot 1296} \,2^3 = \frac{49}{243}\,(R-r)r\mu \geq \frac15(R-r)r\mu
\end{align*}
since $\bar s - r \leq (R-r)/3 < 2(R-r)/3 \leq R-\bar s$. The same argument applies if $u'\leq -1$ on $(r, \bar s)$.

\item Otherwise, there exist points $s_*, s^*$ such that $|u'(s_*)| \leq 1$ and $|u'(s^*)| \geq |a| \geq  2|b|  \geq 4/3$. By the same Modica-Mortola style argument as above, we conclude that
\[
\widetilde F(u, K) \geq 2\sqrt{\mu}r\big(G(4/3)-G(1)\big) + \frac{\lambda r^2}{R} = 2\sqrt\mu r\int_1^{4/3}z^2-1\dz + \frac{\lambda r^2}{R} = \frac{20}{81}\sqrt\mu r+ \frac{\lambda r^2}{R} \geq \frac{\sqrt\mu r}5 + \frac{\lambda r^2}{R}.
\]
\end{itemize}
As we always used the loose lower bound $s\geq r$, the case $K = \{\bar s\}$ with $\bar s\in ((2R+r)/3, R)$ can be treated analogously and yields the same constants.

Overall, if $K\subseteq (r, (2r+R)/3)$ or $K\subseteq((r+2R)/3, R)$, then
\begin{align*}
\widetilde F(u,K) &\geq \min\left\{\frac{\mu r(R-r)}{30}, \: \frac{\sqrt\mu r}{10}, \frac{(R-r)r\mu}5, \:\frac{\sqrt\mu\,r}5\right\}+\lambda\,\frac{r^2}R
    = \min\left\{\frac{\mu r(R-r)}{30}, \: \frac{\sqrt\mu r}{10}\right\} + \lambda\,\frac{r^2}R.
\end{align*}

{\bf Step 5.}
We have shown that, when $K = \emptyset$, $\#K\geq 2$ or $K =\bar s$ with $\bar s\in (r, (2r+R)/3) \cup ((2R+r)/3, R)$ we have
\begin{align*}
\widetilde F(u, K) &\geq \min \left\{\frac{2\lambda r^2}{R},\:\frac{9r(R-r)\mu}{16}, \:\frac{11}{12}\sqrt{\mu}r, \: \frac{\mu r(R-r)}{30} +\frac{\lambda r^2}{R}, \:\frac{\sqrt{\mu}r}{10} +\frac{\lambda r^2}{R}\right\}.
\end{align*}
Thus, in view of Step 1 it suffices to show that 
\begin{align*}
\widetilde F(u, K) \geq \log\left(\frac Rr\right) + \lambda\,\frac{R+r}2 + 1\geq \inf_{(v, L) \in \G_1} F(v, L) + 1
\end{align*}
for appropriate $\mu$ and $\lambda$, i.e.\ 
\begin{align*}
\min \left\{\frac{2\lambda r^2}{R},\:\frac{9r(R-r)\mu}{16}, \:\frac{11}{12}\sqrt{\mu}r, \: \frac{\mu r(R-r)}{30} +\lambda\frac{ r^2}{R}, \:\frac{\sqrt{\mu}r}{10} +\lambda\frac{ r^2}{R}\right\} \geq 1+\log \left(\frac{R}{r}\right) +\lambda \frac{r+R}{2}.
\end{align*}
We solve for $\lambda, \mu$. The first condition is the only one that does not involve $\mu$, so we require
\[
2\lambda \frac{r^2}R \geq 1+ \log\left(\frac Rr\right) + \lambda\frac{r+R}2 \qquad\LRa\quad \lambda \left(2\frac{r^2}R - \frac{r+R}2\right) = \frac{\lambda R}2 \left(4\left(\frac rR\right)^2 - \frac rR - 1\right) \geq 1 + \log\left(\frac Rr\right).
\]
The condition 
\[
4\left(\frac rR\right)^2 - \frac rR - 1 \geq 0 \qquad\Rightarrow \quad 1< \frac Rr\leq \frac{\sqrt{17}-1}2
\]
on the radii $r<R$ is necessary for solvability, and sufficient with
\[
\lambda \geq \frac{2}{R} \frac{1+ \log\left(\frac Rr\right)}{4\left(\frac rR\right)^2 - \frac rR - 1} = \frac{2R\left(1+\log\left(\frac{R}{r}\right)\right)}
{4r^2-rR-R^2}.
\]
Finding admissible $r,R,\lambda$, it is now possible to solve the remaining inequalities for $\mu$ as
\begin{align*}
\mu &\geq \frac{16}{9r(R-r)} \left( 1+\log\left(\frac{R}{r}\right) +\lambda\frac{r+R}{2} \right), \\
\mu &\geq \left[ \frac{12}{11r} \left( 1+\log\left(\frac{R}{r}\right) +\lambda\frac{r+R}{2} \right) \right]^2, \\
\mu &\geq \frac{30}{r(R-r)} \left( 1+\log\left(\frac{R}{r}\right) +\lambda\left(\frac{r+R}{2}-\frac{r^2}{R}\right)
 \right), \\
\mu &\geq \left[ \frac{10}{r} \left( 1+\log\left(\frac{R}{r}\right) +\lambda\left(\frac{r+R}{2}-\frac{r^2}{R}\right)
\right) \right]^2
\end{align*}
and thus ultimately
\begin{align*}
 \mu &\geq \max\biggl\{ \frac{16}{9r(R-r)} \left( 1+\log\left(\frac{R}{r}\right) +\lambda\frac{r+R}{2} \right), \left[ \frac{12}{11r} 
 \left( 1+\log\left(\frac{R}{r}\right) +\lambda\frac{r+R}{2} \right) \right]^2,\\
  &\qquad \frac{30}{r(R-r)} \left( 1+\log\left(\frac{R}{r}\right) +\lambda\left(\frac{r+R}{2}-\frac{r^2}{R}\right) \right),
\left[ \frac{10}{r} \left( 1+\log\left(\frac{R}{r}\right) +\lambda\left(\frac{r+R}{2}-\frac{r^2}{R}\right) \right) \right]^2 \biggl\}.
\end{align*}
If $r,R,\lambda,\mu$ meet all conditions, then 
\[
\widetilde F(u,K) \geq \inf_{(v,L)\in \mathcal G_1} F(v,L)+1
\]
when $K$ does not model a single bend in the center third of the interval $(r,R)$. Further conditions on the parameters will be found below to ensure a lower bound when $K$ models a single bend in the mid third of the interval. This bound required on $R/r$ is stricter than the bound found above; to avoid redundancy, it is the statement given in the Lemma.

{\bf Step 6.} Finally, assume that $K = \{\bar s\}$ with $\bar s\in (a,b) = \big(\frac{2r+R}3, \frac{r+2R}3\big)$. In this case, $\widetilde F$ may take lower values than $F$, albeit not by a large margin. For later use, we obtain a lower bound on $\inf F$ using the previous two steps.
First, we note that $F\geq \widetilde F \geq \inf F + 1$ if $K$ is not a singleton contained in $((2r+R)/3,(r+2R)/3)$. On the other hand, if $K$ is such, then $F(u,K) \geq \lambda \,\frac{2r+R}3$, meaning that $\inf F\geq \lambda\frac{2r+R}3$.

Note that
\begin{align*}
\int_r^R \frac{(u')^2}s\ds &= \int_r^R \frac1s\ds  - \int_r^R\frac{\big(1-(u')^2\big)\sqrt s}{s^{3/2}}\ds\\
    &\leq \log\left(\frac Rr\right) + \left(\int_r^R \frac1{s^3}\ds \right)^{1/2}\left(\int_r^R\big(1-(u')^2\big)^2\,s\ds\right)^{1/2}\\
    &\leq \log\left(\frac Rr\right) + \frac1{\sqrt {2\mu}\,r}\left(\mu\int_r^R\big((u')^2 - 1\big)^2s\ds\right)^{1/2}
\end{align*}
and thus
\begin{align*}
\widetilde F(v, \{\bar s\}) &\geq F\big(v,\{\bar s\}\big) + \lambda\left(\frac{\bar s^2}R - \bar s\right) - \log\left(\frac Rr\right) - \frac1r\,\sqrt{\frac{F\big(v, \{\bar s\}\big)}{2\mu}}\\
    &\geq F\big(v,\{\bar s\}\big) + \lambda\left(\frac{ (2r+R)^2}{9R} - \frac{2r+R}3\right) - \log\left(\frac Rr\right) - \frac1r\,\sqrt{\frac{F\big(v, \{\bar s\}\big)}{2\mu}}\\
    &= F\big(v,\{\bar s\}\big) + \frac{2\lambda}3\left(\frac rR - 1\right)\frac{2r+R}3 - \log\left(\frac Rr\right) - \frac1{\sqrt{2\mu}\,r}\,\sqrt{F\big(v, \{\bar s\}\big)}
\end{align*}
since the discrepancy between $s^2/R$ and $s$ is larger the closer $s$ is to $r>R/2$. 
The inequality is currently pointwise in $v$, so we need to convert it to an absolute inequality for the infimum (which is generally not attained for the same function $v$ by the functionals $F, \widetilde F$). Since
\[
\frac{d}{dz} \left(z - \eps \sqrt z\right) = 1 - \frac\eps2\,z^{-1/2} \geq 0 \qquad\forall\ z\geq \left(\frac\eps2\right)^2,
\]
we have
\begin{align*}
\widetilde F &\geq \inf_{(v,L)\in\mathcal G_1} F(v,L) - \frac23 \lambda \left(1-\frac{r}R\right) \frac{2r+R}3 - \log\left(\frac Rr\right) - \frac1r\sqrt{\frac{\inf F}{2\mu}}\\
    &\geq \inf F - \frac23 \lambda \left(1-\frac{r}R\right)\,\frac{2r+R}3 - \log\left(\frac Rr\right) - \frac1{r\sqrt{2\mu}} \sqrt{\log\left(\frac Rr\right) + \lambda\frac{r+R}2}
\end{align*}
assuming that
\[
\inf_{(u,K)\in \mathcal G_1}F(u,K) \geq \lambda\,\frac{2r+R}3 \geq \frac{1}{8r^2\mu}.
\]
We thus require
\[
\mu \geq \frac{3}{8\lambda\,r^2 (2r+R)}
\]
and 
\[
\frac23 \lambda \left(1-\frac{r}R\right)\,\frac{2r+R}3 + \log\left(\frac Rr\right) + \frac1{r\sqrt{2\mu}} \sqrt{\log\left(\frac Rr\right) + \lambda\frac{r+R}2} \leq \frac17
\]
in order to attain $\widetilde F \geq \inf F - \frac1{7}$. For the latter, we find necessary conditions that
\[
\log\left(\frac Rr\right) < \frac 17\qquad \text{and}\quad \frac23 \lambda \left(1-\frac{r}R\right)\,\frac{2r+R}3 + \log\left(\frac Rr\right) < \frac 17,
\]
i.e.\ that $R< e^{1/7}r$ and 
\[
\lambda \leq \frac{9}2\,\frac{\frac17 - \log(R/r)}{(1-r/R)(2r+R)}.
\]
The first condition is stricter and thus replaces the previous bound on $R/r$. 
If both are satisfied, we can find a final necessary condition on the parameters as
\[
\mu \geq \frac{\log\left(\frac Rr\right) + \lambda\frac{r+R}2}{2r^2\left[\frac 17 - \log\left(R/r\right) - 2\lambda (1-r/R) (2r+R)/9\right]^2}.
\]

{\bf Step 7.} The existence of a minimizer follows by the direct method of the calculus of variations. Specifically, a minimizing sequence $(u_n, K_n)$ has the form $(u_n, \{s_n\})$ with $s_n \in ((2r+R)/3, (r+2R)/3)$ by the previous steps. We can `split' the interval and consider $\widehat F:H^2(0,1)\times H^2(0,1) \times \left[(2r+R)/3, (r+2R)/3\right] \to \R$ with
\begin{align*}
\widehat F(u_1, u_2, s) &= 
 (s-r)^{-3}\int_0^1|u_1''|^2 \big(r + (s-r)\sigma\big)\d\sigma + (R-s)^{-3}\int_0^1|u_2''|^2 \big(s + (R-s)\sigma\big)\d\sigma\\
    &\qquad + (s-r)^{-1} \int_0^1 \frac{(u_1')^2}{r+(s-r)\sigma}\d\sigma + (R-s)^{-1}\int_0^1 \frac{(u_2')^2}{s+(R-s)\sigma}\d\sigma\\
    &\qquad + (s-r)\mu\int_0^1\Big(1- (s-r)^{-2}(u_1')^2\Big)^2  (r+(s-r)\sigma)\d\sigma\\
    &\qquad+ (R-s)\mu\int_0^1\Big(1- (R-s)^{-2}(u_2')^2\Big)^2  (s+(R-s)\sigma)\d\sigma+ \lambda s &
\end{align*}
if $u_1(0) = u_2(1) = 0$ and $u_1(1) = u_2(0)$ and $+\infty$ otherwise. The functional $\widehat F$ is lower semi-continuous with respect to the weak convergence of $u_1, u_2$ in $H^2(0,1)$ and the convergence of $s$ to a value {\em inside} $(r,R)$ by standard results, in particular the compact embedding yielding strong convergence of $u_1',u_2'$ in $L^4(0,1)$. In particular, $\widehat F$ has a minimizer $(\bar u_1, \bar u_2, \bar s)$.

By design, 
\[
u:(r,R)\to \R, \qquad u(s) = \begin{cases} u_1\left(\frac{s-r}{\bar s - r}\right) & s\in (r,\bar s)\\
u_2\left(\frac{s-\bar s}{R-\bar s}\right) &s\in (\bar s, R)\end{cases}
\]
is a minimizer of $F$. An analogous construction is possible for $\widetilde F$.
\end{proof}

\subsection{Auxiliary Results: Straight lines in annuli}

\begin{proof}[Proof of Lemma \ref{lemma angles}]
We denote the distance from the origin to the line $L$ by $x$. Depending on $x$, $L$ may intersect different numbers of annuli. We will prove the claim by analyzing the cases illustrated in Figure \ref{figure angles situations} for $x$ separately, moving from larger to smaller $x$. Without loss of generality, we may assume that $L$ is parallel to the $x$-axis.

\begin{figure}
\begin{tikzpicture}[scale = 1]
\filldraw[fill = red!30](2,0) arc (0:180:2);
\filldraw[fill = blue!30](5/3,0) arc (0:180:5/3);
\filldraw[fill = orange!30](4/3,0) arc (0:180:4/3);
\filldraw[fill = white](1,0) arc (0:180:1);
\draw[very thick](-2.1,2.15) -- ++ (4.2,0);
\end{tikzpicture}
\begin{tikzpicture}[scale = 1]
\filldraw[fill = red!30](2,0) arc (0:180:2);
\filldraw[fill = blue!30](5/3,0) arc (0:180:5/3);
\filldraw[fill = orange!30](4/3,0) arc (0:180:4/3);
\filldraw[fill = white](1,0) arc (0:180:1);
\fill[very thick, fill = red!50, domain = 66:112] (0,0) -- plot ({.5*cos(\x)}, {.5*sin(\x)});
\draw[thick, black!50](0,0) -- ({sqrt(4 - 1.85^2}, 1.85);
\draw[thick, black!50](0,0) -- ({-sqrt(4 - 1.85^2}, 1.85);
\draw[very thick](-2.1,1.85) -- ++ (4.2,0);
\end{tikzpicture}
\begin{tikzpicture}[scale = 1]
\filldraw[fill = red!30](2,0) arc (0:180:2);
\filldraw[fill = blue!30](5/3,0) arc (0:180:5/3);
\filldraw[fill = orange!30](4/3,0) arc (0:180:4/3);
\filldraw[fill = white](1,0) arc (0:180:1);
\fill[very thick, fill = red!50, domain = 47:63] (0,0) -- plot ({-.5*cos(\x)}, {.5*sin(\x)});
\fill[very thick, fill = red!50, domain = 47:63] (0,0) -- plot ({.5*cos(\x)}, {.5*sin(\x)});
\fill[very thick, fill = blue!50, domain = 65:115] (0,0) -- plot ({.5*cos(\x)}, {.5*sin(\x)});
\draw[thick, black!50](0,0) -- ({sqrt(4 - 1.5^2}, 1.5);
\draw[thick, black!50](0,0) -- ({-sqrt(4 - 1.5^2}, 1.5);
\draw[thick, black!50](0,0) -- ({sqrt((5/3)^2 - 1.5^2}, 1.5);
\draw[thick, black!50](0,0) -- ({-sqrt((5/3)^2 - 1.5^2}, 1.5);
\draw[very thick](-2.1,1.5) -- ++ (4.2,0);
\end{tikzpicture}
\vspace{3mm}

\begin{tikzpicture}[scale = 1]
\filldraw[fill = red!30](2,0) arc (0:180:2);
\filldraw[fill = blue!30](5/3,0) arc (0:180:5/3);
\filldraw[fill = orange!30](4/3,0) arc (0:180:4/3);
\filldraw[fill = white](1,0) arc (0:180:1);

\fill[very thick, fill = blue!50, domain = 43:60] (0,0) -- plot ({.5*cos(\x)}, {.5*sin(\x)});
\fill[very thick, fill = red!50, domain = 35:43] (0,0) -- plot ({.5*cos(\x)}, {.5*sin(\x)});
\fill[very thick, fill = blue!50, domain = 43:60] (0,0) -- plot ({-.5*cos(\x)}, {.5*sin(\x)});
\fill[very thick, fill = red!50, domain = 35:43] (0,0) -- plot ({-.5*cos(\x)}, {.5*sin(\x)});
\fill[fill = orange!50, domain = 60:120] (0,0) -- plot ({-.5*cos(\x)}, {.5*sin(\x)});

\draw[thick, black!50](0,0) -- ({sqrt(4 - 1.15^2}, 1.15);
\draw[thick, black!50](0,0) -- ({-sqrt(4 - 1.15^2}, 1.15);
\draw[thick, black!50](0,0) -- ({sqrt((5/3)^2 - 1.15^2}, 1.15);
\draw[thick, black!50](0,0) -- ({-sqrt((5/3)^2 - 1.15^2}, 1.15);
\draw[thick, black!50](0,0) -- ({sqrt((4/3)^2 - 1.15^2}, 1.15);
\draw[thick, black!50](0,0) -- ({-sqrt((4/3)^2 - 1.15^2}, 1.15);
\draw[very thick](-2.1,1.15) -- ++ (4.2,0);
\end{tikzpicture}
\begin{tikzpicture}[scale = 1]
\filldraw[fill = red!30](2,0) arc (0:180:2);
\filldraw[fill = blue!30](5/3,0) arc (0:180:5/3);
\filldraw[fill = orange!30](4/3,0) arc (0:180:4/3);
\filldraw[fill = white](1,0) arc (0:180:1);

\fill[very thick, fill = blue!50, domain = 30:39] (0,0) -- plot ({.5*cos(\x)}, {.5*sin(\x)});
\fill[very thick, fill = red!50, domain = 26:30] (0,0) -- plot ({.5*cos(\x)}, {.5*sin(\x)});
\fill[fill = orange!50, domain = 40:59] (0,0) -- plot ({.5*cos(\x)}, {.5*sin(\x)});
\fill[very thick, fill = blue!50, domain = 30:39] (0,0) -- plot ({-.5*cos(\x)}, {.5*sin(\x)});
\fill[very thick, fill = red!50, domain = 26:30] (0,0) -- plot ({-.5*cos(\x)}, {.5*sin(\x)});
\fill[fill = orange!50, domain = 40:59] (0,0) -- plot ({-.5*cos(\x)}, {.5*sin(\x)});

\draw[thick, black!50](0,0) -- ({sqrt(4 - .85^2}, .85);
\draw[thick, black!50](0,0) -- ({-sqrt(4 - .85^2}, .85);
\draw[thick, black!50](0,0) -- ({sqrt((5/3)^2 - .85^2}, .85);
\draw[thick, black!50](0,0) -- ({-sqrt((5/3)^2 - .85^2}, .85);
\draw[thick, black!50](0,0) -- ({sqrt((4/3)^2 - .85^2}, .85);
\draw[thick, black!50](0,0) -- ({-sqrt((4/3)^2 - .85^2}, .85);
\draw[thick, black!50](0,0) -- ({sqrt((3/3)^2 - .85^2}, .85);
\draw[thick, black!50](0,0) -- ({-sqrt((3/3)^2 - .85^2}, .85);
\draw[very thick](-2.1,.85) -- ++ (4.2,0);
\end{tikzpicture}
\begin{tikzpicture}[scale = 1]
\filldraw[fill = black!20](2,0) arc (0:180:2);
\filldraw[fill = black!25](5/3,0) arc (0:180:5/3);
\filldraw[fill = black!10](4/3,0) arc (0:180:4/3);
\filldraw[fill = white](1,0) arc (0:180:1);

\draw[black!50](0,0) -- ({sqrt(4 - 1.5^2}, 1.5);
\draw[thick](0,0) -- ({sqrt((5/3)^2 - 1.5^2}, 1.5);
\draw[red,  thick](0,0) -- (0,1.5);
\draw[red,  thick](0,0) -- ({sqrt((5/3)^2 - 1.5^2)},1.5);
\node[left, red] at(.1, .75){$x$};
\node[right, red] at(.28, .75){$r$};
\draw[red, thick] (0, 1.3) arc(270:360:.2);
\fill[red](.07,1.43) circle(.2mm);
\fill[red, very thick, domain = 63:90, samples = 30] (0,0) -- plot ({.3*cos(\x)}, {.3*sin(\x)}) -- (0,0);
\draw[ thick, red](0, 1.5) -- ({sqrt((5/3)^2 - 1.5^2)}, 1.5);
\draw[very thick](-2.1,1.5) -- ++ (4.2,0);
\end{tikzpicture}

\caption{\label{figure angles situations}
As the distance of the straight line from the origin varies, the sets $A, B, C$ as in Figure \ref{figure angles} can be empty or have one or two connected components, depending on which annuli the line intersects. The four regimes where the line intersects at least one annulus can also be noted in the second line of Figure \ref{figure angles}. The final plot illustrates the auxiliary triangle by which we obtain expressions for the volume of $A, B,C$, with the solid red angle being precisely $\cos^{-1}(x/r)$.
}
\end{figure}
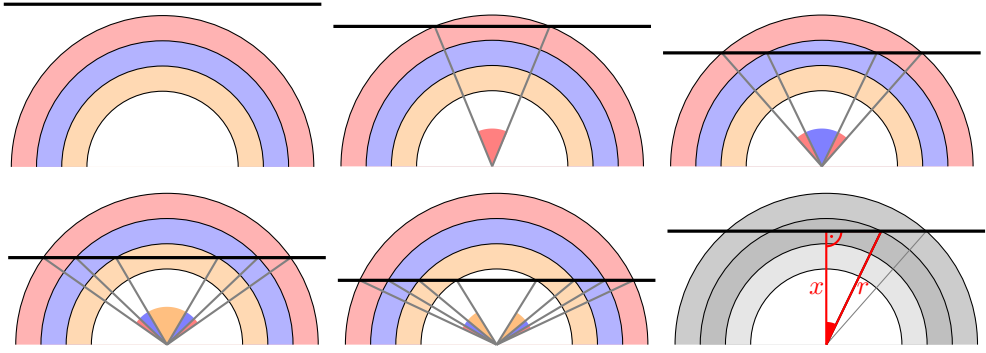

{\bf Step 1.} First, when $b \leq x,$ we have
$0 = |B| \leq 6 (|A|+|C|)$, which immediately settles two cases in Figure \ref{figure angles situations}. Now, we consider the case where $a \leq x<b.$ Forming an auxiliary triangle with the $y$-axis as in Figure \ref{figure angles situations} and using symmetry, we find that
\begin{align*}
\frac{|B|}{|A|+|C|} = \frac{|B|}{|A|} = \frac{ \cos^{-1}(\frac{x}{b}) }{ \cos^{-1}(\frac{x}{R}) - \cos^{-1}(\frac{x}{b})} &=  \frac{ 1 }{ \frac{ \cos^{-1}(\frac{x}{R})}{\cos^{-1}(\frac{x}{b})} - 1 } = \frac{1}{k(x) - 1}, 
\end{align*}
where
\[
k(x) :=  \frac{ \cos^{-1}(\frac{x}{R})}{\cos^{-1}(\frac{x}{b})}.
\]
We now show that $k$ is increasing. Its derivative satisfies
\[
k'(x) = \frac{  \sqrt{R^2 - x^2} \cos^{-1}\left(\frac{x}{R}\right) - 
\sqrt{b^2 - x^2} \cos^{-1}\left(\frac{x}{b}\right)}{\sqrt{b^2 - x^2} \sqrt{R^2 - x^2}\left(\cos^{-1}\left(\frac{x}{b}\right)\right)^2}.
\]
It suffices to show that:
\begin{equation}\label{eq:k}
\sqrt{R^2 - x^2} \cos^{-1}\left(\frac{x}{R}\right) - 
\sqrt{b^2 - x^2} \cos^{-1}\left(\frac{x}{b}\right) >0.
\end{equation}
To prove this we claim that for a fixed $0<a<x<b,$ the following function is increasing for $b \leq z$:
\[
F_x(z) = \sqrt{z^2 - x^2} \cos^{-1}\left(\frac{x}{z}\right).
\]
We note that we can rewrite \eqref{eq:k} as $F_x(R) -F_x(b) > 0.$ 
We have
\[
F_x'(z) = \frac{x}{z} +
\frac{z \cos^{-1} (\frac{x}{z}) }{\sqrt{z^2 - x^2}}  >0,
\]
which shows that $F_x(z)$ is increasing and $F_x(R) > F_x(b)$. 
So far we have shown that for $a<x<b,$ $k(x) > k(a),$ thus:
\[
\frac{|B|}{|A|+|C|} \leq \frac{1}{k(a) - 1} = \frac{\cos^{-1}(\frac{a}{b})}{ \cos^{-1}(\frac{a}{R}) - \cos^{-1}(\frac{a}{b})}. 
\]
We will return to this expression below.

{\bf Step 2.} Now, suppose that $0<x<a$. Again, this step encompasses two of the five cases in Figure \ref{figure angles situations}. As above, we observe that
\begin{align*}
|B| &= \cos^{-1}\left(\frac{x}{b}\right) - \cos^{-1}\left(\frac{x}{a}\right)\\
|A| &= \cos^{-1}\left(\frac{x}{R}\right) - \cos^{-1}\left(\frac{x}{b}\right).
\end{align*}
Therefore
we note that
\begin{align*}
\frac{|B|}{|A| + |C|} &\leq \frac{|B|}{|A|}
 \leq \frac{\cos^{-1}(\frac{x}{b}) - \cos^{-1}(\frac{x}{a})}{ \cos^{-1}(\frac{x}{R}) - \cos^{-1}(\frac{x}{b})}.
\end{align*}
Thus, we define
\[
h(x):=\frac{\cos^{-1}(\frac{x}{b}) - \cos^{-1}(\frac{x}{a})}{ \cos^{-1}(\frac{x}{R}) - \cos^{-1}(\frac{x}{b})},
\]
and note that it suffices to bound $h$ from above by 6.
To this end, we show that $h$ is an increasing function  and thus $h(x)\leq h(a),$ for $x \in (0,a)$. We write $h(x)$ as
$h(x) = \frac{f(x)}{g(x)}$ and using  L'Hopital's rule for monotonicity \cite{estrada2017hopital}, we note that $f(0) = g(0) =0$ and we show that $l(x) = \frac{f'(x)}{g'(x)}$ is an increasing function. 
We have:
\[
l(x) = 
\frac{f'(x)}{g'(x)} = 
\frac{\frac{1}{\sqrt{a^2-x^2}} -  \frac{1}{\sqrt{b^2-x^2}} } { \frac{1}{\sqrt{b^2-x^2}} - \frac{1}{\sqrt{R^2-x^2}}   }.
\]
To show that $l$ is an increasing function we now need to show that $l'(x) > 0.$ After simplification, we see that:
\begin{align*}
l'(x) &= 
x\left(\sqrt{b^2-x^2}-\sqrt{a^2-x^2}\right) \left(\sqrt{R^2-x^2}-\sqrt{a^2-x^2}\right) \\
&\hspace{2cm}\cdot \frac{\sqrt{a^2-x^2}\sqrt{b^2-x^2}+\sqrt{a^2-x^2}\sqrt{R^2-x^2}+\sqrt{b^2-x^2}\sqrt{R^2-x^2}}{\left(a^2-x^2\right)^{3/2}\sqrt{b^2-x^2}\sqrt{R^2-x^2} \left(\sqrt{R^2-x^2} - \sqrt{b^2-x^2}\right)}.
\end{align*}
So we see that $l'(x)>0$ for $x \in (0,r).$ Using L'Hopital's rule for monotonicity \cite{estrada2017hopital}, we conclude that both $l$ and $h$ are increasing in $(0,a)$. In particular, $h$ attains its supremum at $x = a$ where
\[
\frac{|B|}{|A|} = h(a) = \frac{\cos^{-1}(a/b) - \cos^{-1}(1)}{\cos^{-1}(a/R) - \cos^{-1}(a/b)} = \frac{\cos^{-1}(a/b)}{\cos^{-1}(a/R) - \cos^{-1}(a/b)}
\]
as in Step 1.

{\bf Step 3.} Finally, it remains to bound the shared expression for the maximum. Let $\alpha=\frac{R}{r}>1$, then $R=\alpha r$, and we have
\begin{align*}
\frac{\cos^{-1}(\frac{a}{b})}{ \cos^{-1}(\frac{a}{R}) - \cos^{-1}(\frac{a}{b})} = \frac{
\cos^{-1}\left(\frac{2+\alpha}{1+2\alpha}\right)
}{
\cos^{-1}\left(\frac{2+\alpha}{3\alpha}\right)
-
\cos^{-1}\left(\frac{2+\alpha}{1+2\alpha}\right)
}.
\end{align*}
Since $ \cos^{-1}\left(\frac{2+\alpha}{1+2\alpha}\right) = \cos^{-1}\left(\frac{2+\alpha}{3\alpha}\right)-\cos^{-1}\left(\frac{2+\alpha}{1+2\alpha}\right) =0 $ at $\alpha =1,$ we can apply L'Hopital's rule for monotonicity and observe that:
\[
\frac{\left(\cos^{-1}\left(\frac{2+\alpha}{1+2\alpha}\right)\right)'}{\left( \cos^{-1}\left(\frac{2+\alpha}{3\alpha}\right)-\cos^{-1}\left(\frac{2+\alpha}{1+2\alpha}\right) \right)'} =   \frac{3\alpha}{\sqrt{6\alpha^2+9\alpha+3}-3\alpha}, 
\]
and
\[
\frac{d}{d \alpha}  \frac{3\alpha}{\sqrt{6\alpha^2+9\alpha+3}-3\alpha} =
\frac{3(3\alpha+2)}{2\sqrt{\frac{2\alpha^2}{3}+\alpha+\frac{1}{3}}\left(\sqrt{6\alpha^2+9\alpha+3}-3\alpha\right)^2} >0.
\]
Therefore the supremum occurs at infinity (or $\alpha = R/r$ as large as admitted) and we have:
\begin{align*}
\sup_{\alpha \in (1,\infty)}
\frac{
\cos^{-1}\left(\frac{2+\alpha}{1+2\alpha}\right)
}{\cos^{-1}\left(\frac{2+\alpha}{3\alpha}\right)-\cos^{-1}\left(\frac{2+\alpha}{1+2\alpha}\right)}
    &= \lim_{\alpha \to \infty} \frac{
\cos^{-1}\left(\frac{2+\alpha}{1+2\alpha}\right)
}{\cos^{-1}\left(\frac{2+\alpha}{3\alpha}\right)-\cos^{-1}\left(\frac{2+\alpha}{1+2\alpha}\right)} \\
    &= \frac{\cos^{-1}\left(1/2\right)}{\cos^{-1}\left(1/3\right) - \cos^{-1}\left(1/2\right)}  \\
    &= \frac{\pi/3}{\cos^{-1}(1/3) - \pi/3} \approx 5.7 < 6.
\end{align*}
If $R/r\leq 3/2$, we evaluate the function numerically for an estimate.
\end{proof}

In the previous Lemma, we considered a single subspace $L$ and the quotient of the measures of the set of `good' angles (angles where $L\cap (B_R\setminus B_r)$ lies in the central sub-annulus) and the set of `guaranteed bad' angles (angles where $L\cap B_R\setminus B_r$ lies in one of the other two sub-annuli). Other angles for which $L$ did not intersect the annulus $B_R\setminus B_r$ were not considered at all since those could be bad for the subspace $L$, but could be `cured' by another line $L'$. In the next lemma, we extend the result to a finite union of subspaces.

\begin{proof}[Proof of Lemma \ref{lemma multiple lines}]
We can modify the collection of lines $\mathcal L$ in any way which does not decrease the `good set'. This allows us to drop all lines from $\mathcal L$ which do not intersect $B_b\setminus B_a$ as they can only decrease the good set, not increase it. Similarly, we may drop all lines passing directly through the origin. As $\mathcal L$ is locally finite, this leaves us a finite collection of lines $\mathcal L' = \{\ell_1, \dots, \ell_N\}$.

Again, we consider the projection $\pi:\R^2\to S^1$ and for a line $\ell$, we denote 
\[
G(\ell) = \pi\big(\ell \cap (B_b\setminus B_a)\big), \qquad I(\ell) = \pi\big(\ell\cap (B_R\setminus B_r)\big), \qquad B(\ell) = I(\ell)\setminus G(\ell)
\]
and note that
\[
G(\L) \subseteq G(\L') = \bigcup_{i=1}^N \left(G(\ell_i) \setminus \bigcup_{j\neq i} I(\ell_j)\right).
\]
As can be seen e.g.\ in Figure \ref{figure angles situations}, the set $I(\ell)$ may consist of one or two connected components (intervals) on the circle. For greater flexibility, we `decouple' the connected components associated to the same line and note that for the purpose of proof we can modify individual intervals rather than lines in a way that increases the good set. By symmetry exploited in the proof of Lemma \ref{lemma angles}, the same bounds apply to the individual intervals. Specifically, assume from now on that we are given a finite collection $\mathcal I$ of intervals $I_i, G_i$ in $S^1$ such that $I_i = G_i \bigcup B_i$ in $S^1$ such that $|G_i| \leq \zeta |B_i|$ for some universal $\zeta$. We define the good set of the collection analogously
\[
G(\mathcal I) = \bigcup_{i=1}^M \left(G_i \setminus \bigcup_{j\neq i} I_j\right). 
\]

Obviously, if $G_i \subseteq \bigcup_{j\neq i} I(\ell_j)$, then $I_i$ does not actually contribute to $G(\I)$ and we can remove $I_i$ from the collection $\I$. As $\I$ is finite, after a finite number of steps this yields a sub-collection $\I'$ such that for every $I_i\in \I'$, there exists $\theta_i \in G_i \setminus \bigcup_{j\in \I', j\neq i} I_j$, i.e.\ every interval has a good point which is not contained in any other interval.

Since the sets $I_i$ are intervals, this means that any $\theta\in S^1$ can lie in at most two different intervals: One coming from the right and one coming from the left. This is easy to see on $\mathbb R$: If $x\in I_1 \cap I_2\cap I_3 =  (a_1, b_1)\cap  (a_2, b_2) \cap (a_3, b_3)$ and without loss of generality $a_1 = \min a_i$, $b_2 = \max b_i$, then $(a_3,b_3)\subseteq (a_1, b_2) = (a_1,b_1)\cup (a_2,b_2)$ since the intervals share a point. Thus, the collection would violate the `unique point' property. The intervals in $S^1$ lift to proper intervals in $\mathbb R$ since they have length $<\pi$ (see Figure \ref{figure angles situations}) and the lifting can only underestimate overlap, not overestimate it, so the same statement applies.  
We conclude that
\begin{align*}
\H^1(G(\I')) &= \H^1\left(\bigcup_i\left(G_i \setminus \bigcup_{j\neq i} I_j\right)\right)
    \leq \sum_i \H^1\big(G_i\big)
    \leq \zeta \sum_i \H^1(B_i)\\
    &= \zeta \int_{S^1} \sum_i \chi_{B_i} \d\H^1 \leq \zeta \int_{S^1} 2\,\chi_{\bigcup_i B_i}\d\H^1 
    \leq 2\zeta \,\H^1\left(\bigcup_{i} B_i\right)
    \leq 2\zeta\,\H^1\big(B(\mathcal I')\big).
\end{align*}
Here we used the monotonicity and sub-additivity of measures, the fact that a point is counted at most twice in the union and the fact that $\bigcup_{i} B_i\subseteq B(\I')$. Hence
\[
\H^1(B(\I')) = 2\pi - \H^1(G(\I')) \geq 2\pi - 2\zeta\,\H^1(B(\I')) 
\]
which finally yields
\[
\H^1(B(\L)) \geq \H^1(B(\L')) \geq \H^1(B(\L')) \geq \frac{2\pi}{2\zeta+1}.
\]
The claim follows using $\zeta = 6$ in general and $\zeta = 3$ if $R/r\leq 3/2$ as per Lemma \ref{lemma angles}.
\end{proof}

\subsection{Proof of Main Result}
We conclude by establishing an energy gap between the class of surfaces modeled by graphs of Barron functions and more general graph surfaces. 

\begin{proof}[Proof of Theorem \ref{theorem bending}]
{\bf Step 1: Radially symmetric functions.} From \eqref{eq hessian polar}, we find that
\[
\|D^2 u\|^2 = u_{rr}^2 + \frac{|u_r|^2}{r^2}
\]
for radially symmetric functions and $\|D^2u\|^2 \geq u_{rr}^2$ for all functions.
Similarly, we have $\nabla u = u_r \,e_r + u_\phi \,e_\phi$ and thus
\[
\| \nabla u \otimes \nabla u - g \|^2 = \| \nabla u \otimes \nabla u - e_r \otimes e_r \|^2 \geq  \big|u_r^2 -1\big|^2
\]
since $e_r\otimes e_\phi, e_r\otimes e_r$ and $e_\phi\otimes e_\phi$ are orthonormal with respect to the Frobenius inner product. The inequality is an equality if and only if $u_\phi = 0$, i.e.\ it holds as an equality precisely for radially symmetric functions.
We conclude that
\[
E_{\mu,\lambda,g}(u) = \int_{\Omega \setminus K} |u_{rr}|^2 + \frac{|u_r|^2}{r^2} \dx + \mu \int_\Omega \big|u_r^2 -1\big|^2 \dx +  \lambda \,\H^1(\mathcal K)
\]
for radially symmetric functions $u$. Radially symmetric functions of finite energy have the form $u(x) = v(|x|)$ where $v\in H^1(r,R)$ and $v\in H^2$ except at a finite set of radii $K_v$, i.e.\
\[
E_{\mu,\lambda,g}(u) = 2\pi \left(\int_{(r,R)\setminus K_v} \left(|v''|^2 + \frac{|v'|^2}{s^2}\right)s\ds + \mu \int_r^R\big|(v')^2-1\big|^2s\ds + \lambda \sum_{s\in K_v}s \right)= 2\pi\,F(v, K_v)
\]
where $F$ is as in Lemma \ref{lemma slicing}, with the same parameters $\lambda, \mu$ as in $E_{\mu,\lambda,g}$ and thus
\[
\inf_{(u, K) \in \mathcal G} E_{\mu,\lambda,g}(u,K) \leq 2\pi\,\inf_{(v, K)\in \mathcal G_1} F(v,K).
\]
The inequality may be strict if minimizers break the radial symmetry, which may happen for the non-convex energy functional.

{\bf Step 2. Barron functions.} Analogous to Step 1, we note that
\begin{align*}
E_{\mu,\lambda,g}(u, K) &\geq \int_{\Omega \setminus K} |u_{rr}|^2 \dx + \mu \int_\Omega \big|u_r^2 -1\big|^2 \dx +  \lambda \,\H^1( K)
\end{align*}
for arbitrary pairs $(u,K)$.
Given an angle $\theta$, we define
\[
K_\theta := \left\{\rho \in (r,R) : \rho(\cos\theta,\sin\theta) \in K \right\},
\qquad
u^\theta : (r,R) \to \mathbb{R}, \qquad  u^\theta(\rho) := u(\rho\cos\theta,\rho\sin\theta)
\]
such that $u_r = (u^\theta)'$ and $u_{rr} = (u^\theta)''$. 

Assume that $L$ is an affine line which intersects a disk $B_R(0)$ at a distance $d>0$. Without loss of generality, we assume that $L = \{(x,y) : y=d\}$ and parametrize the line by $x(\theta) = d\, (\tan \theta, 1)$ with $\theta \in (\theta_1,\theta_2)$ such that the angle $x(\theta)$ makes with the positive $y$-axis is $\theta$. Then
\begin{align*}
\H^1(L\cap \Omega) &= \int_{\theta_1}^{\theta_2} \big|x'(\theta)\big| \d\theta = \int_{\theta_1}^{\theta_2} d\,\sqrt{\big(\tan'\theta\big)^2+0}\d\theta = \int_{\theta_1}^{\theta_2} d\,\big(1+\tan^2\theta\big)\d\theta\\
    &= \int_{\theta_1}^{\theta_2} \frac{|x(\theta)|^2}{d}\d\theta\geq \int_{\theta_1}^{\theta_2} \frac{|x(\theta)|^2}R\d\theta = \int_0^{2\pi} \sum_{s\in (L\cap B_R)_\theta} \frac{s^2}R\d\theta
\end{align*}
The final formulation naturally applies to annular domains as well as countable unions of lines. Thus, if $u$ is general and $K$ is a countable union of lines, we have
\begin{align*}
E_{\mu,\lambda,g}(u,K) &\geq \int_{0}^{2\pi} \int_{(r,R)\setminus K_\theta}\big|(u^\theta)''\big|^2 s + \mu\,\Big|\big((u^\theta)'\big)^2 -1\Big|^2 s \ds + \lambda\sum_{s\in K_\theta} \frac{s^2}R\d\theta\\
    &= \int_0^{2\pi} \widetilde F\big(u^\theta, K_\theta)\d\theta\\
    &= \int_{G(K)} \widetilde F\big(u^\theta, K_\theta)\d\theta + \int_{B(K)} \widetilde F\big(u^\theta, K_\theta)\d\theta
\end{align*}
where $G(K), B(K)$ denote the set of good and bad angles of the union of lines defining $K$. By Lemma \ref{lemma slicing} and the first step of the proof, we have
\begin{align*}
E_{\mu,\lambda,g}(u,K) &\geq \H^1\big(B(K)\big) \big(\inf F + 1\big) + \H^1\big(G(K)\big) \left(\inf F - \frac1{7}\right)\\
    &= 2\pi \,\inf F + \left(\H^1(B(K)) - \frac{\H^1(G(K))}{7}\right)\\
    &\geq \inf_{(u,K)\in \mathcal G_1}E(u,K) + \left(\H^1(B(K)) - \frac{\H^1(G(K))}{7}\right).
\end{align*}
Finally, we use Lemma \ref{lemma multiple lines} to conclude that
\[
E_{\mu,\lambda,g}(u,K) \geq \inf_{(u,K)\in \mathcal G_1}E(u,K) + 2\pi\left(\frac{1}{7} - \frac{6}{7}\frac{1}{7}\right) = \inf_{(u,K)\in \mathcal G_1}E(u,K) + \frac{2\pi}{7^2}.
\]
The proof was formulated assuming that $K$ is a union of affine lines. While this is not guaranteed for finite energy pairs $(u,K)$ with $u\in \B$, if $u\in \mathcal B\cap SBV^2$, then we can replace $K$ by a union of affine lines $K'$ with potentially lower measure as in Theorem \ref{theorem structure barron sbv2}. As this only decreases the energy, we can conclude the proof.
\end{proof}

\section{Acknowledgements}

The authors gratefully acknowledge financial support by grant NSF DMS 2424801.

\appendix 
\section{On the structure of Barron functions}\label{appendix barron structure}

We begin by proving that the distributional Hessian of a Barron function is well-defined and a measure by showing that the partial derivatives of Barron functions are functions of bounded variation.

\begin{proof}[Proof of Lemma \ref{lemma bounded variation gradient}]
If $u\in \B$, there exists a sequence $u_n(x) = \sum_{k=1}^n a_k \sigma(w_k\cdot x+b_k)$ such that
\begin{enumerate}
    \item $\sum_{k=1}^n |a_k| \,\|w_k\|_2 \leq \|u\|_\B$ and
    \item $u_n\to u$ in $L^1(B_R(0))$
\end{enumerate}
by the direct approximation theorem for Barron functions (see e.g.\ \cite[Theorem 3.8]{barron_new}). We note that
\[
D^2 u_n = \sum_{k=1}^n a_k \,\frac{w_k\otimes w_k}{\|w_k\|}\,\H^{n-1}|_{\{x: w\cdot x+b = 0\}}
\]
and thus 
\begin{align*}
|D^2u_n|(B_R(0)) &= \sum_{k=0}^n |a_k| \,\|w_k\|\,\H^{d-1}\big(\{x: w_k\cdot x+b_k = 0\}\cap B_R(0)\}\big)\\
    &\leq \omega_{d-1}R^{d-1}\sum_{k=0}^n|a_k| \,\|w_k\| \leq \omega_{d-1}R^{d-1} \|u\|_\B.
\end{align*}
By \cite[Theorem 1.9]{giusti1984minimal}, the inequality survives in the limit $n\to\infty$.
\end{proof}

Key to this observation is that the Hessian grows {\em linearly} in $w$ in the expression above, not quadratically as one might naively expect. This is easy to see in one dimension by the positive one-homogeneity 
\[
\frac{d^2}{dx^2} \sigma(wx) = \frac{d^2}{dx^2}\, w\sigma(x)  = w\,\sigma''= w\,\delta_0
\]
or by smooth approximation $\sigma_\eps(wx)$ where $w^2\sigma_\eps''(wx) \to w\,\delta_0$ since a factor of $w$ is lost to a change of variables under the integral.

Throughout the proof of Theorem \ref{theorem structure barron sbv2}, we use the same notations and ideas as in \cite{barron_new}. Specifically, we use that $f$ can be represented as
\[
f(x) = \int_{S^2} \sigma\big(w\cdot (x,1)\big) \d\mu_{w}
\]
for a signed finite Borel measure $\mu$ on $S^2\subseteq \R^3$ such that $\|\mu\|_{TV} = \|f\|_\B$.

\begin{proof}[Proof of Theorem \ref{theorem structure barron sbv2}]
{\bf Step 1. Barron functions and positive homogeneity.} In this step, we show that a general Barron function can be written as a sum of simpler Barron functions, which we can further analyze assuming that the regular part of the Hessian is $L^2$-integrable.

Since the total variation measure $|\mu|$ is finite, there exist at most countably many atoms $w_1, w_2, \dots$ of $|\mu|$ since the set of atoms $\mathcal A$ satisfies
\[
\mathcal A = \bigcup_{k\in\mathbb N_0} \left\{x \in S^2 : |\mu(x)| \geq \frac1k\right\}
\]
and every one of the sets in the union must be finite by the bound on $|\mu(S^2)|$. Thus $\mu = \tilde \mu + \sum_{i=1}^\infty m_i\,\delta_{w_i}$ for some $m_i\in \R$ and distinct $w_i$ such that the remainder measure $|\tilde\mu|$ has no atoms and $\|\mu\| = \|\tilde\mu\| + \sum_{i=1}^\infty|m_i|$. 

We can further decompose $\tilde\mu = \mu^* + \sum_{i=1}^\infty \tilde\mu_i$ where $\tilde\mu_i$ is a measure supported on a great circle in $S^2$. The measure $|\mu^*|$ does not give mass to any great circle in $S^2$. This can be proved much like the previous result, using that
\[
|\tilde \mu|(S^2) \geq |\tilde\mu|\left(\bigcup_{j=1}^\infty C_j\right) = \sum_{j=1}^\infty |\tilde\mu|(C_j)
\]
for any collection of great circles $C_1, C_2, \dots$ in $S^2$ since $|\tilde\mu|(C_j\cap C_k) = 0$ for any two distinct great circles $C_j, C_k$ in $S^2$ since $C_j\cap C_k$ is finite and $|\tilde\mu|$ has no atoms.

We consider the different parts of the decomposition separately. The functions
\begin{align*}
f_i(z) &= \int_{S^2}\sigma(w\cdot z)\d (m_i\delta_{w_i}) = m_i\,\sigma(w_i\cdot z)
\end{align*}
are single neuron activations. They are smooth except on the plane $\langle w_i\rangle^\bot = \{z: w_i\cdot z= 0\}$ which intersects the `physical' plane $\{z = (x,1) : x\in\R^2\}$ in a line (or not at all). The Hessian $D^2f_i$ concentrates on the line and vanishes everywhere else. 

On the other hand, if $C_i$ denotes the circle which is given positive mass by $|\tilde\mu_i|$, then the functions $\tilde f_i(z) = \int_{S^2}\sigma(w\cdot z) \d\tilde\mu_i$ are $C^1$-smooth with gradient
\[
\nabla \tilde f_i(z) = \int_{S^2} \chi_{\{w: w\cdot z>0\}}\,w \d\tilde\mu_i
    = \int_{\{w\in S^2 : w\cdot z>0\}} w\d\tilde\mu_i
\]
except on the line $L_i$ which is orthogonal to the line $C_i$ lies in.   In particular, if the functions $\tilde f_i$ are piecewise linear in $\R^3$, they are in fact linear. The line $L_i$ intersects the `physical plane' in a single point $(x_i, 1)$ (or not at all) and the function $\tilde f_i(x-x_0)$ is positively homogeneous of degree one and $C^1$-smooth except at the origin.

{\bf Step 2. Blow-up analysis.}
We note that for any $x_0\in \R^2$, the dominated convergence theorem implies that
\begin{align*}
\lim_{\lambda\searrow 0}\frac{f(x_0+\lambda x) - f(x_0)}\lambda &= \lim_{\lambda\searrow 0} \int_{S^2} \frac{\sigma\big(w\cdot(x_0,1) +  \lambda\,w\cdot(x,0)\big) - \sigma(w\cdot(x_0,1))}\lambda \d\mu\\
    &= \int_{S^2} \lim_{\lambda\searrow 0}\frac{\sigma\big(w\cdot(x_0,1) +  \lambda\,w\cdot(x,0)\big) - \sigma(w\cdot(x_0,1))}\lambda \d\mu\\
    &= \int_{\{w\in S^2 : w\cdot(x_0,1) > 0\}} \frac{\lambda \,w\cdot(x,0)}\lambda\d\mu + \int_{\{w \in S^2 : w\cdot(x_0,1) = 0\}} \sigma\big(w\cdot(x,0)\big)\d\mu\\
    &= \left(\int_{\{w\in S^2 : w\cdot(x_0,1) > 0\}} w\lambda\d\mu\right)\cdot \begin{pmatrix}x\\0\end{pmatrix} + \int_{\{w \in S^2 : w\cdot(x_0,1) = 0\}} \sigma\big(w\cdot(x,0)\big)\d\mu.
\end{align*}
with dominating function $w\mapsto \|w\|\,\|x\| = \|x\|$.
The convergence holds pointwise and in $L^2(\Omega)$, again by the dominated convergence theorem (with dominating function $\|\mu\|_{TV}\|x\|$). The first term describes a linear function, the second a positively one-homogeneous Barron function may be non-zero if and only if $|\mu|$ assigns positive mass to the great circle $\{w\in S^2: w\cdot (x_0,1) = 0\}$.

Note that the absolutely continuous part of the Hessian of a positively one-homogeneous function is positively homogeneous of degree $-1$, so in two dimensions, it {\em cannot} be square integrable unless it vanishes almost everywhere. This is the case only for piecewise linear functions.

If the limiting positively one-homogenous function fails to lie in $H^2(B_r(0))$, then clearly $\|D^2_x(f(x_0+\lambda x) - f(x_0))/\lambda\|_{L^2(B_r(0))}$ must diverge: If it remained bounded, then the norm $\|v\|_{L^2(B_r)} + \|D^2v\|_{L^2(B_r)}$, which is equivalent to the standard $H^2$-norm (see \cite[Aufgabe 6.4]{dobrowolski2010angewandte} or use \cite[Theorem 9.1]{gilbarg2015elliptic}) would remain bounded as well. This would mean that the limit is attained not just $L^2$-strongly but also $H^2$-weakly and thus that the limit would lie in $H^2$.

Due to two-dimensional scaling, this means that the Barron function cannot have a blow-up that is not piecewise linear at any point. More specifically, the absolutely continuous part of the Hessian transforms as
\[
u_\lambda(x) = \frac1\lambda\,u(\lambda x), \qquad \nabla u_\lambda(x) = (\nabla u)(\lambda x), \qquad D^2 u_\lambda(x) = \lambda\,(D^2u)(\lambda x)
\]
and so 
\[
\int_{B_r} \|D^2u_\lambda\|^2\dx = \int_{B_r}\|D^2u\|^2 (\lambda x) \,\lambda^2 \dx = \int_{B_{\lambda r}} \|D^2 u\|^2\dx. 
\]
Thus, if $\lim_{\lambda\searrow0}\frac{f(x_0+\lambda x) - f(x_0)}\lambda$ converges to a positively one-homogeneous function which is {\em not} merely a finite sum of ReLU activations and a linear function, then $u\notin SBV^2_{loc}(\R^2)$. Thus $\tilde f_i$ is linear for all $i$ (and could be represented e.g.\ by a sum of two ReLU activations).

Thus, we see that inside $\Omega$, we have $f(x) = f^*(x) + \sum_{i=1}^\infty m_i\sigma\big(w_i\cdot (x,1)\big)$ with $f^*\in H^2(\Omega)$. The function $f^*$ is solely responsible for the absolutely continuous part of the Hessian and the minimal set $K'$ on which the measure part of $D^2f$ concentrates is $\bigcup_{\{i \in \mathbb N : m_i \neq 0\}} \{w_i\cdot(x,1) =0\}\cap \Omega\}$.

Since $\Omega$ is open, for every $x\in \Omega$ there exists $r>0$ such that $B_r(x)\subseteq\Omega$. If $L$ is any line that intersects $B_{r/2}(x)$, then $\H^1(L\cap B_{r/2}(x))\geq r/2 = \dist(\partial B_r(x), \partial B_{r/2}(x))$. Thus, there are at most a finite number of lines $\{w_i\cdot(x,1) = 0\}$ which intersect $B_{r/2}(x)$, meaning that the union of lines to $K'$ is {\em locally finite}.
\end{proof}

\bibliographystyle{alpha}
\bibliography{bibliography.bib}
\end{document}